\documentclass[11pt]{amsart}
\usepackage{pifont}
\usepackage{amsthm}
\usepackage{amsfonts}
\usepackage{amssymb}
\usepackage[mathscr]{euscript}
\usepackage[all]{xy}
\usepackage[dvips]{graphics}
\usepackage{amsmath}
\usepackage[dvips,final]{graphicx}
\usepackage[dvips]{geometry}
\usepackage{color}
\usepackage{epsfig}
\usepackage{latexsym}
\usepackage{pstricks}
\usepackage{multirow}
\usepackage{rotating}
\usepackage{hyperref}
\usepackage{enumerate}

\newtheorem{theorem}{Theorem}
\newtheorem{lemma}[theorem]{Lemma}
\newtheorem{proposition}[theorem]{Proposition}
\newtheorem{corollary}[theorem]{Corollary}

\theoremstyle{definition}
\newtheorem{definition}{Definition}[section]

\theoremstyle{remark}
\newtheorem{remark}{Remark}[section]

\title[Three Loop Invariants]{The Three Loop Isotopy and Framed Isotopy Invariants of Virtual Knots}
\author{Micah W. Chrisman and Heather A. Dye}

\begin{document}
\begin{abstract} This paper introduces two virtual knot theory ``analogues'' of a well-known family of invariants for knots in thickened surfaces: the Grishanov-Vassiliev finite-type invariants of order two. The first, called the three loop isotopy invariant, is an invariant of virtual knots while the second, called the three loop framed isotopy invariant, is a regular isotopy invariant of framed virtual knots. The properties of these invariants are investigated at length.  In addition, we make precise the informal notion of ``analogue''. Using this formal definition, it is proved that a generalized three loop invariant is a virtual knot theory analogue of a generalization of the Grishanov-Vassiliev invariants of order two.
\end{abstract}
\maketitle
\section{Introduction}
\subsection{Overview} An interpretation of virtual knots is that they are knots in thickened oriented surfaces modulo stabilization and destabilization.  Given a knot in a thickened surface $K$, a virtual knot $K'$ to which $K$ is stably equivalent, and a virtual knot invariant $v$, one easily obtains an invariant of knots in a thickened surface via the map $K \to v(K')$.
\newline
\newline
\noindent This observation becomes more interesting when viewed from the opposite direction. If one starts with an invariant of knots in a thickened surface, under what circumstances does there exist a virtual knot ``analogue''? One may then study the virtual knot invariant by studying its ``analogue''. The advantage of this viewpoint is that one can exploit powerful topological and geometric tools which are available for knots in 3-manifolds. 
\newline
\newline
The present paper considers this approach for a well-known family of invariants of knots in thickened surfaces: the Grishanov-Vassiliev finite-type invariants of order two \cite{GrVa}. We introduce two ``analogues'' of these invariants.  The first, called the three loop isotopy invariants, are  finite-type invariants of order two for virtual knots.  The second, called the three loop framed isotopy invariants, are regular virtual isotopy invariants of framed virtual knots. They are ``analogues'' of the Grishanov-Vassiliev invariants in the sense that they have identical Gauss diagram formulae. The difference is that the topological enhancements of the Gauss diagram in the Grishanov-Vassiliev invariants (which are either conjugacy classes in $\pi_1(\Sigma)$ or elements of $H_1(\Sigma;\mathbb{Z})$) are replaced with combinatorial enhancements in the three loop/three loop framed isotopy invariants. The three loop isotopy invariants are enhanced with a \emph{relative weight} of regions in a Gauss diagram. The three loop framed isotopy invariants are enhanced with the \emph{weight} (or \emph{index}) \cite{allison} of a crossing.
\newpage
\noindent With some additional work, the informal notion of ``analogue'' can be made precise. For this purpose we provide a generalization of the Grishanov-Vassiliev invariants and a generalization of the three loop invariant. For a given closed oriented surface $\Sigma$, the generalized Grishanov-Vassiliev invariant is an invariant $\Phi[\Sigma]$ of knots in $\Sigma \times I$ valued in an abelian group $\mathscr{A}(\Sigma)$. The generalized three loop invariant is a virtual knot invariant $\phi$ valued in an abelian group $\mathscr{A}$. We will prove that $\phi$ is a virtual knot analogue of $\Phi[\Sigma]$ for all $\Sigma$ in the sense that the following diagram commutes.
\[
\xymatrix{
\mathbb{Z}[\mathscr{K}(\Sigma)] \ar[r]^{\Phi[\Sigma]} \ar[d] & \mathscr{A}(\Sigma) \ar[d] \\
\mathbb{Z}[\mathscr{K}]\ar[r]_{\phi} & \mathscr{A} \\
}
\]    
Here, $\mathbb{Z}[\mathscr{K}(\Sigma)]$ is the free abelian group generated by the set of equivalence classes of knots in $\Sigma \times I$ and $\mathbb{Z}[\mathscr{K}]$ is the free abelian group generated by equivalence classes of virtual knots. The left vertical arrow represents the usual projection of knots in $\Sigma \times I$ to virtual knots.  The right vertical arrow, which will be defined later, is constructed using intersection theory.
\newline
\newline
The three loop/three loop framed isotopy invariants have many other interesting properties.  We investigate their properties in terms of finite-type invariants, connected sum, and geometric symmetries (inverses, horizontal mirror images, vertical mirror images, and crossing vitualizations).
\newline
\newline
The organization of the present paper is as follows. In Section \ref{vkt}, we provide a brief review of virtual knot theory. In Section \ref{sec_phi_ijk}, we give a combinatorial definition of the three loop invariant and provide examples of its computation. In Section \ref{sec_phi}, we give a combinatorial definition of the three loop framed isotopy invariant and provide examples of its computation. Knowledge about the Grishanov-Vassiliev invariants of order two is not a prerequisite in Sections \ref{sec_phi_ijk} and \ref{sec_phi}. In Section \ref{sec_properties}, we establish the behavior of the three loop/three loop framed isotopy invariants with respect to finite-type, connected sum, and geometric symmetries. In Section \ref{sec_gvinariant}, we give a precise definition of ``analogue'' and prove that a generalization of the three loop invariants is a virtual knot analogue of a generalization of the Grishanov-Vassiliev invariant. Finally, in Section \ref{sec_conc}, we state a problem on extending our approach to universal finite-type invariants of knots in thickened surfaces.

\subsection{Review of Virtual Knot Theory} \label{vkt}

\subsubsection{Gauss Diagrams and Virtual Knots} A virtual link is an equivalence class of virtual link diagrams \cite{kvirt}. Two virtual link diagrams are said to be equivalent if they are related by a finite sequence of the extended Reidemeister moves (see Figure \ref{fig:rmoves}). Let $\mathscr{K}$ denote the set of virtual knots (i.e. one component virtual links). The set of framed virtual knots is the set of equivalence classes of virtual knot diagrams under all extended Reidemeister moves except classical Reidemeister 1. Let $\mathscr{K}^{fr}$ denote this set of equivalence classes. 

\begin{figure}[htb] 
\[
\begin{array}{|c|} \hline
\epsfysize = 1.5 in
\epsffile{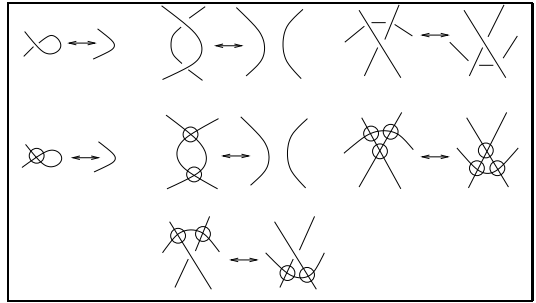}\\
\hline
\end{array}
\]
\caption{Reidemeister moves}
\label{fig:rmoves}
\end{figure}

\noindent Let $K$ be an oriented, virtual knot diagram. As such it is an immersion $K:S^1 \to \mathbb{R}^2$ such that each double point is drawn as a classical crossing or a virtual crossing. Draw a copy of $ S^1 $ in $ \mathbb{R}^2 $ that is oriented in the counterclockwise direction. The \emph{Gauss diagram} \cite{gpv} of $K$ is constructed from the immersion as follows. For each double point of the map that is a classical crossing, we connect the pre-images of $K$ in the circle by a chord. Each chord is directed as an arrow from the over-crossing to under-crossing arc. In addition, each chord is given a sign: $\oplus$ if the corresponding crossing is positively oriented and $\ominus$ if the corresponding crossing is negatively oriented. The directed, signed chords are referred to as \emph{arrows}.  We will denote by $G_K$ the Gauss diagram corresponding to $K$. Two Gauss diagrams are said to be equivalent if they are related by a finite sequence of diagrammatic moves. The diagrammatic moves are analogues of the Reidemeister moves. A sufficient set of these moves is shown in Figure \ref{fig:analog} \cite{gpv}.

\begin{figure}[htb] 
\[
\begin{array}{|c|c|} \hline
\begin{array}{ccc}
\multicolumn{3}{c}{\underline{\text{Reidemeister 1 Move:}}} \\
\begin{array}{c} \scalebox{.25}{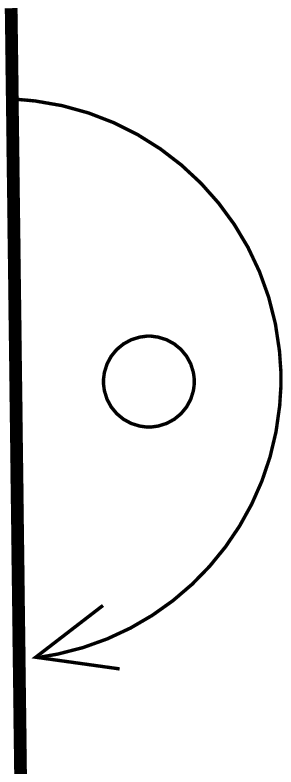} \end{array} & \leftrightarrow & \begin{array}{c} \scalebox{.25}{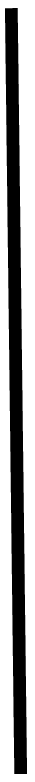} \end{array}
\end{array} &
\begin{array}{ccc}
\multicolumn{3}{c}{\underline{\text{Reidemeister 2 Move:}}} \\
\begin{array}{c} \scalebox{.25}{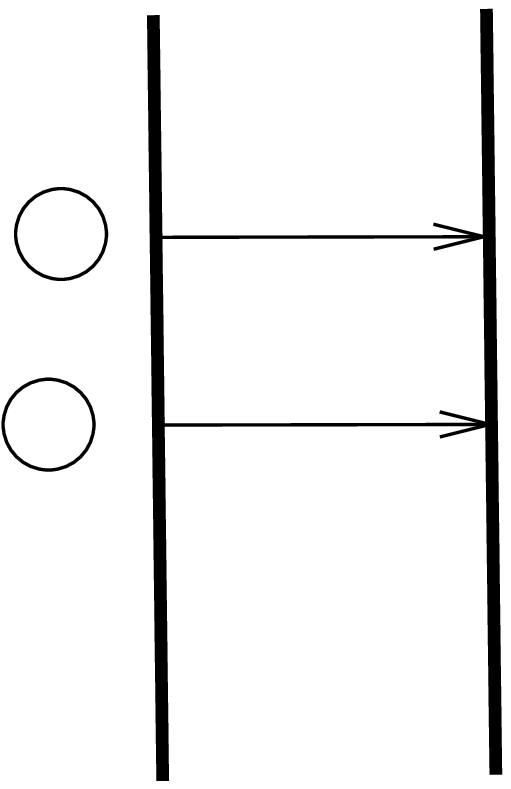} \end{array} & \leftrightarrow & \begin{array}{c} \scalebox{.25}{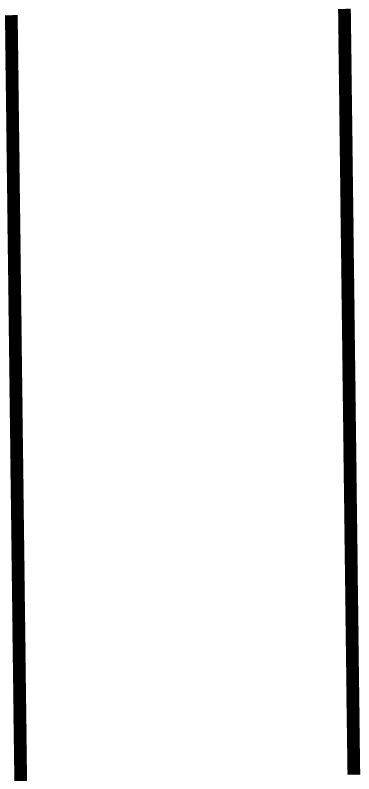} \end{array}
\end{array} \\ \hline
\multicolumn{2}{|c|}{
\begin{array}{ccc}
\multicolumn{3}{c}{\underline{\text{Reidemeister 3 Move:}}} \\
\begin{array}{c} \scalebox{.25}{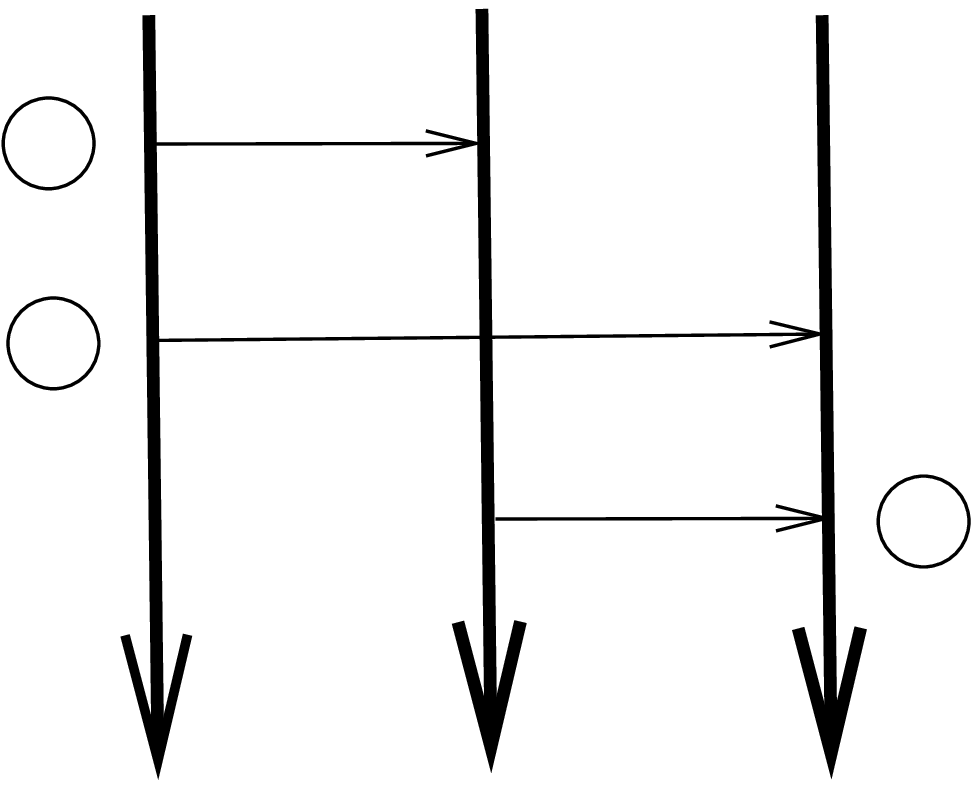} \end{array} &  \leftrightarrow & \begin{array}{c} \scalebox{.25}{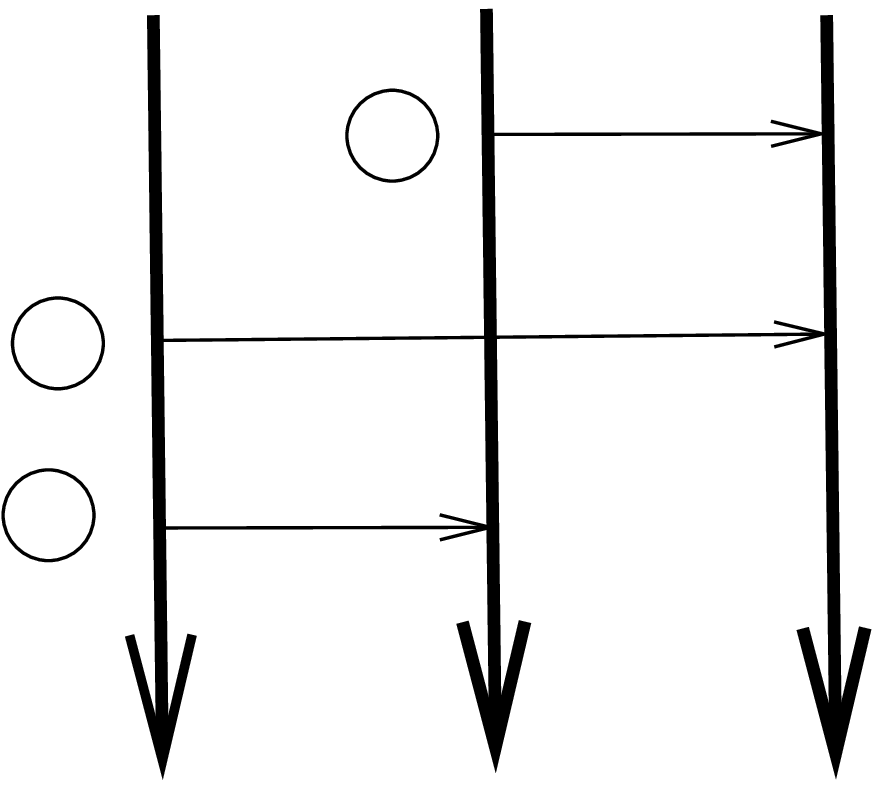} \end{array}
\end{array}} \\ \hline
\end{array}
\]
\caption{Reidemeister moves for Gauss Diagrams.} 
\label{fig:analog}
\end{figure}

\subsubsection{The Index and Smoothing}
We review the definition of the weight (also called the index) of a crossing. The weight of a crossing has been used to construct several invariants of virtual knots, (see \cite{cheng_knot,cheng_link,heather_linking,heather_smoothed,allison,affine_index,fol_kauff,im_lee}). The weight of a crossing is also related to the parity of a crossing \cite{paritymant}. By abuse of notation, we denote both a crossing in a knot diagram and the corresponding arrow in a Gauss diagram as $c$. Without loss of generality, we assume that the arrows are drawn as chords of a circle. Two arrows intersect if they intersect as figures in the interior of the disk.  Let $N_c$ denote the set of arrows that intersect the arrow $c$. Note that $ N_c$ may be the empty set. Given that the arrows are drawn on the interior of the diagram, the \textit{intersection number of $x$ with $c$} (denoted $ int_c (x)$) is defined diagrammatically in Figure \ref{fig:intx}. If the configuration of the two arrows is as appears on the left of Figure \ref{fig:intx}, then $int_c(x)=1$. If the configuration of the two arrows is as appears on the right of Figure \ref{fig:intx}, then $int_c(x)=-1$. 

\begin{figure}[htb] 
\scalebox{0.3}{
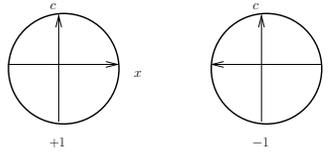}
\caption{The value of $ int_c (x) $}
\label{fig:intx}
\end{figure}
\hspace{1cm}
\newline
\newline
\noindent The \emph{weight} $w(c)$ of the crossing $c$ is defined to be:
\begin{equation}
w(c) = \sum_{x \in N_c} sign (x) \text{} int_c (x).
\end{equation}

\begin{proposition} \label{weightprop} Weights of crossings satisfy the following properties. 
\begin{enumerate}
\item The non-classical extended Reidemeister moves preserve all weights.
\item If a crossing is introduced via a Reidemeister 1 move, the weight of the introduced crossing is zero. 
\item For the two crossings depicted in a Reidemeister 2 move, the weight of the crossings are equal. Deleting the two crossings does not change the weights of any remaining crossings.
\item For the three crossings depicted in a Reidemeister 3 move, the weight of a crossing on the left hand side is the same as the weight of the corresponding crossing on the right hand side of the move. The weights of all other crossings are preserved.  
\end{enumerate}
\end{proposition}
\begin{proof} See \cite{affine_index, heather_smoothed}. 
\end{proof} 

\noindent Lastly, we define the \emph{oriented smoothing of a crossing}. The oriented smoothing of a crossing is the diagrammatic modifcation in a small ball in $\mathbb{R}^2$ depicted on the left hand side of Figure \ref{fig:vertsmooth}. By the \emph{oriented smoothing of an arrow} in a Gauss diagram, we mean the diagrammatic modification in $D^2$ depicted on the right hand side of Figure \ref{fig:vertsmooth}. 

\begin{figure}[htb]
\[
\begin{array}{|c|c|} \hline
\xymatrix{\begin{array}{c} \scalebox{.2}{\psfig{figure=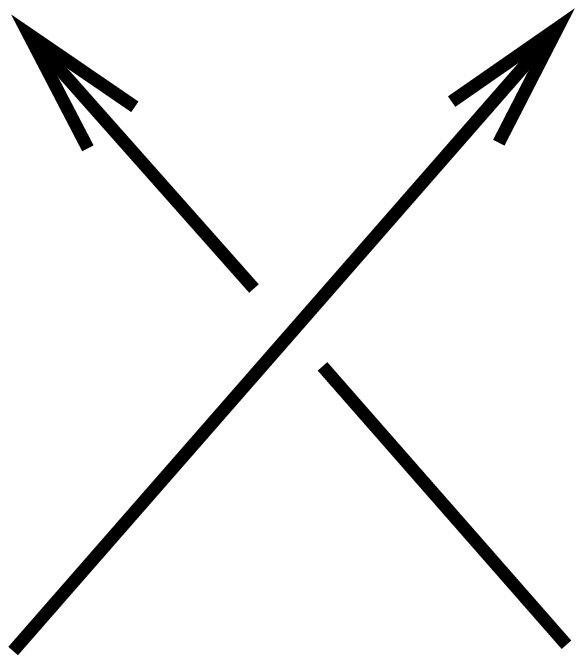}} \end{array} \ar[r] & \begin{array}{c} \scalebox{.2}{\psfig{figure=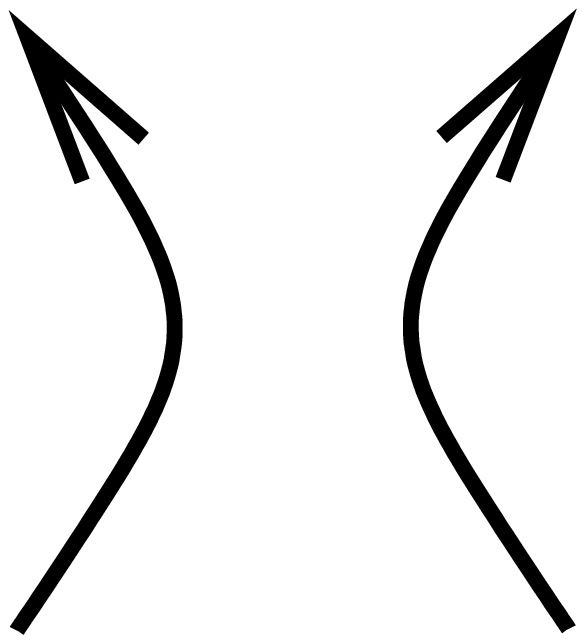}} \end{array} & \ar[l] \begin{array}{c} \scalebox{.2}{\psfig{figure=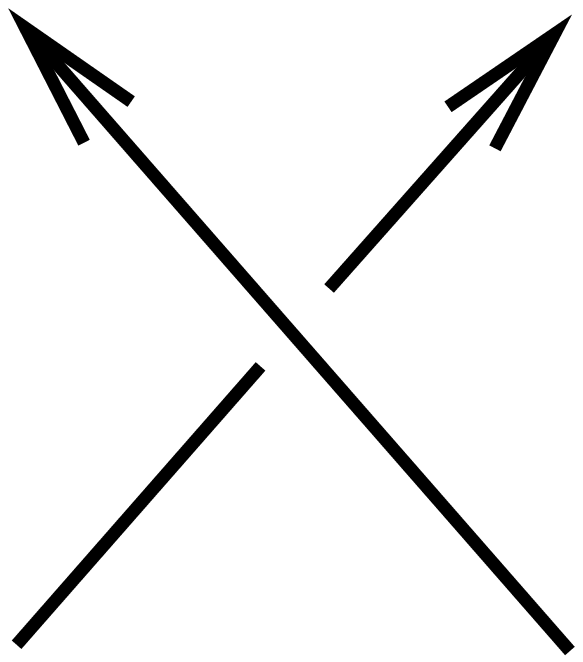}} \end{array}} &  \xymatrix{\begin{array}{c} \scalebox{.2}{\psfig{figure=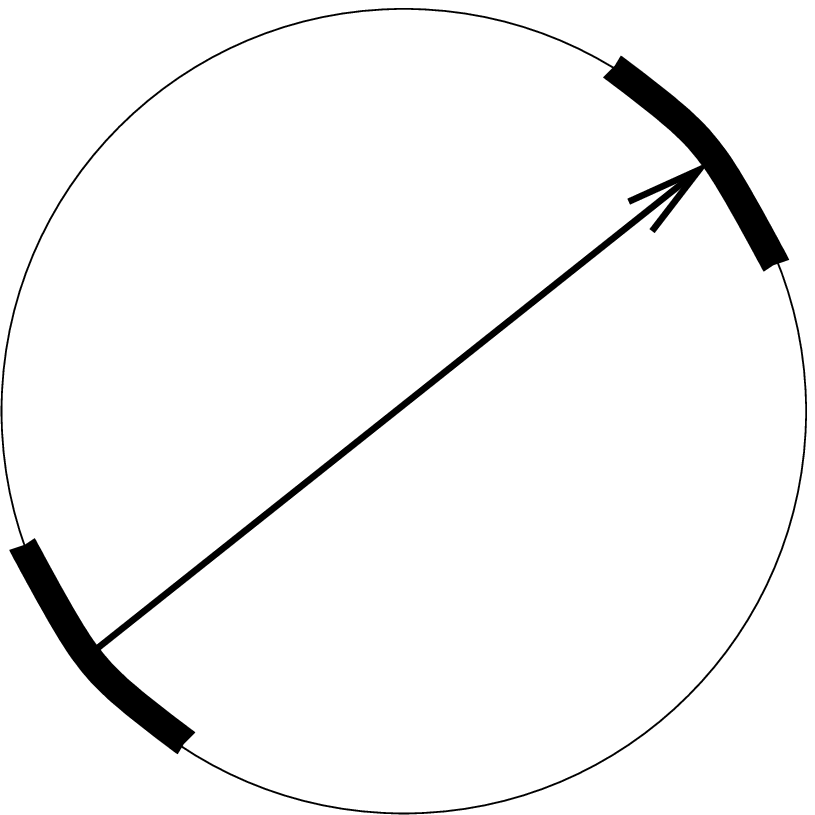}} \end{array} \ar[r] & \begin{array}{c} \scalebox{.2}{\psfig{figure=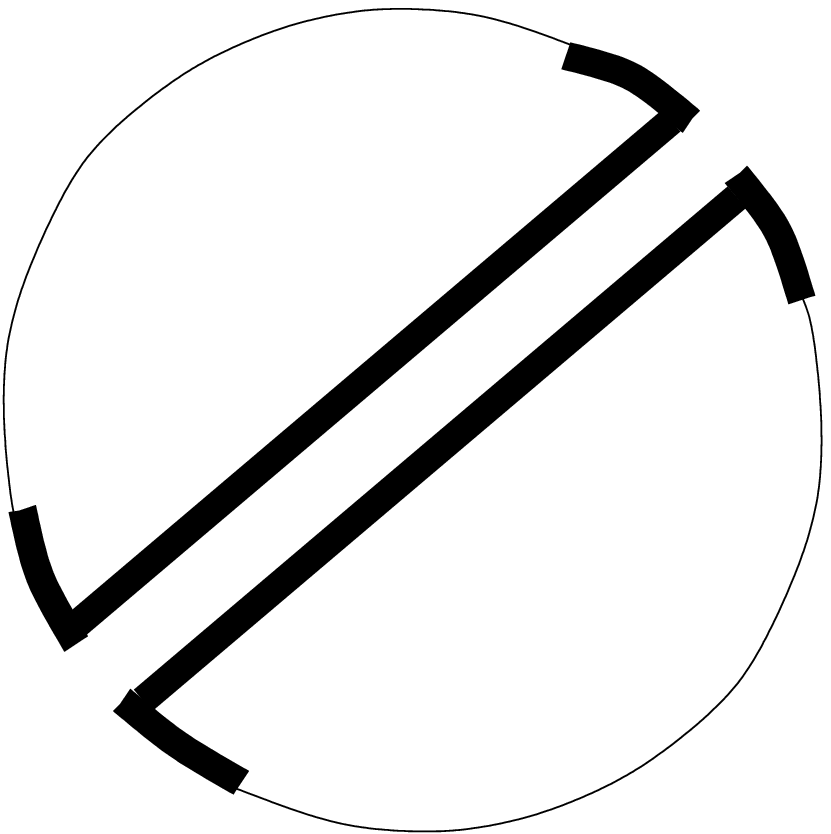}} \end{array}}\\ \hline
\end{array}
\] 
\caption{Oriented smoothing of a crossing (left) and oriented smoothing of an arrow (right).}
\label{fig:vertsmooth}
\end{figure}

\section{The Three Loop Isotopy Invariant} \label{sec_phi_ijk}

\noindent The three loop invariant is an invariant which is given by a Gauss diagram formula \cite{pv} with labeled pairs of regions (see below for definition). Much work has been done on finding Gauss diagram formulae for classical and virtual knot invariants \cite{chmutov_polyak,ckr,micah_mant,gibson_ito,BP,fol_kauff}. The aim of this section is to define the three loop invariant, prove that it is an invariant of virtual knots, and provide some examples of its computation. 

\subsection{Definition of $\phi_{i,j,k}$} Let $D$ be a Gauss diagram of a virtual knot $K$.  Let $D_{||}'$ be a subdiagram of $D$ having two non-intersecting arrows. Let $A_1$, $A_2$, and $A_3$ be the connected components of the Gauss diagram after the oriented smoothing along both of the arrows of $D_{||}'$. Let $C_{i,j}$ denote the set of arrows with one endpoint in $A_i$ and one endpoint in $A_j$. If $A_i$ and $A_j$ are separated be a single arrow, let $a_{ij}$ denote the arrow which separates them.  If $A_i$ and $A_j$ are not separated by an arrow, let $a_{ij}$ denote either of the two arrows. Define:
\[
w_{ij}=\left|\sum_{c \in C_{i,j}} sign(c) \text{ }int_{a_{ij}}(c) \right|.
\] 
Note that $w_{i,j}=w_{j,i}$ and that $w_{i,j}$ is independent of the choice of $a_{ij}$.
\newline
\newline
We assign the weights $w_{ij}$ to the pairs of regions of $D_{||}'$ according to the corresponding diagram in Figure \ref{fig_region_label}. Let $F_{i,j,k}$ denote the formal sum of Gauss diagrams in Figure \ref{fig_phi_ijk_defn}. Define $sign(D_{||}')$ to be the product of the signs of two arrows of $D_{||}'$. Let $\left<F_{i,j,k},D_{||}'\right>=sign(D_{||}')$ if $D_{||}'$ is equivalent as a Gauss diagram with labeled pairs of regions to one of the diagrams in the formal sum $F_{i,j,k}$. Otherwise, $\left<F_{i,j,k},D_{||}'\right>=0$.
The \emph{three loop invariant} is defined to be:
\[
\phi_{i,j,k}(K)=\sum_{D_{||}' \subset D} \left<F_{i,j,k},D_{||}'\right>.
\]
 
\begin{figure}[htb]
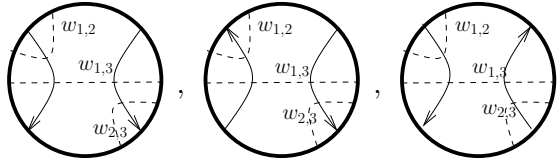

\[
\begin{array}{c}
\scalebox{.4}{\input{reg_lab_1.pstex_t}}
\end{array},\begin{array}{c}
\scalebox{.4}{\input{reg_lab_2.pstex_t}}
\end{array},\begin{array}{c}
\scalebox{.4}{\input{reg_lab_3.pstex_t}}
\end{array}
\]
\caption{The assignment of weights to pairs of regions in the Gauss diagram.} \label{fig_region_label}
\end{figure}

\begin{figure}[htb]
\[
F_{i,j,k}=\begin{array}{c}
\scalebox{.4}{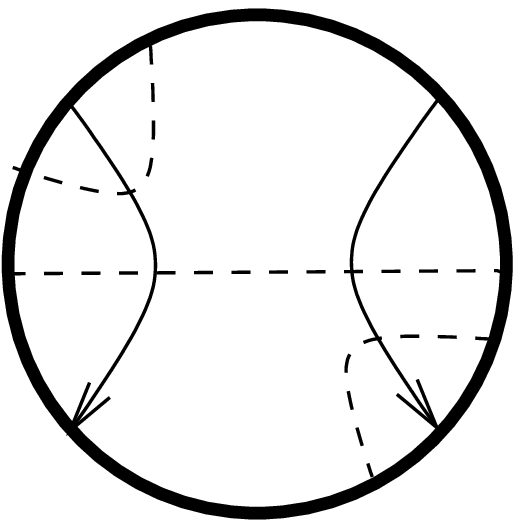}
\end{array}+\begin{array}{c}
\scalebox{.4}{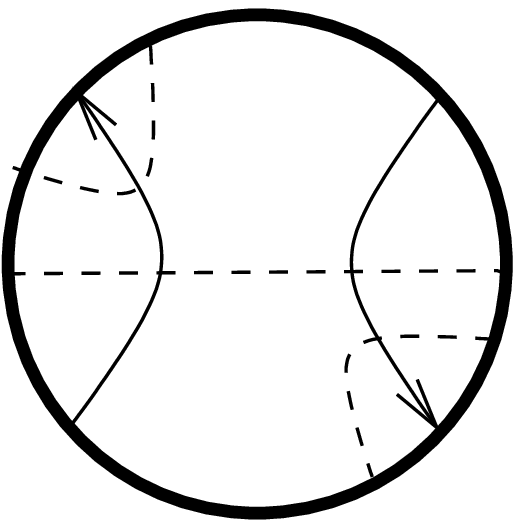}
\end{array}+\begin{array}{c}
\scalebox{.4}{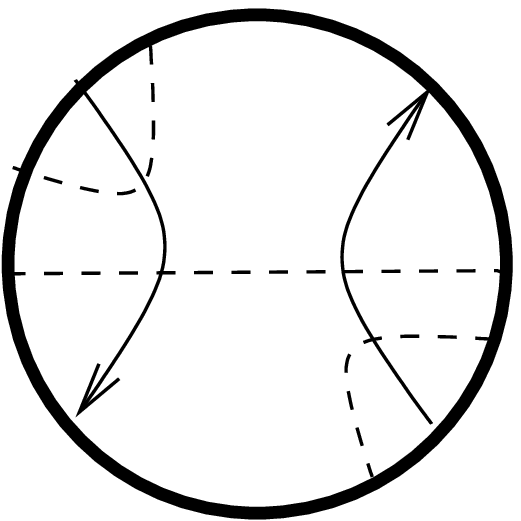}
\end{array}
\]
\caption{The formal sum defining the virtual knot invariant $\phi_{i,j,k}$.}\label{fig_phi_ijk_defn}
\end{figure}

\begin{theorem} \label{lemm_phi_ijk} Let $i,j,k \in \mathbb{N} \cup \{0\}$ be distinct non-negative integers (i.e. $i \ne j \ne k \ne i$).  Then $\phi_{i,j,k}:\mathbb{Z}[\mathscr{K}] \to \mathbb{Z}$ is an invariant of virtual knots.
\end{theorem}
\begin{proof} It will be shown that $\phi_{i,j,k}$ is invariant under the three Reidemeister moves. 
\newline
\newline
Let $D$ be a Gauss diagram of a virtual knot $K$ representing the left hand side of a Reidemeister move in Figure \ref{fig:analog}. Let $D_{||}'$ be a labeled subdiagram of $D$ consisting of two non-intersecting arrows.
\newline
\newline
\underline{Reidemeister 1:} If none of the arrows of $D_{||}'$ participates in the move, then $D_{||}'$ is a subdiagram of the Gauss diagram of the right hand side of the move. If one of the arrows participates in an Reidemeister $1$ move, it follows that two of $w_{1,2}$, $w_{1,3}$, $w_{2,3}$ are zero.  Since at most one of $i,j,k$ is zero, it follows that $D_{||}'$ is not counted by $F_{i,j,k}$. Since $D_{||}'$ does not appear on the right hand side of the Reidemeister move, it follows that $\phi_{i,j,k}$ is invariant under the Reidemeister 1 move.
\newline
\newline
\underline{Reidemeister 2:} Recall that the weight of an arrow which is not involved in a Reidemeister move is unchanged by the move. Thus, if $D_{||}'$ contains none of the arrows in the move, it follows that $D_{||}'$ also appears as a subdiagram on the right hand side of the move. If one of the the arrows of $D_{||}'$ is involved in the move then there is a subdiagram $D_{||}''$ also on the left hand side having the same labels but with opposite sign. Thus $\left<F_{i,j,k},D_{||}' \right>+\left<F_{i,j,k},D_{||}'' \right>=0$.
\newline
\newline
Suppose then that both arrows of $D_{||}'$ are involved in the move.  Since the arrows of $D_{||}'$ are non-intersecting, smoothing along both crossings yields three curves, one of which is a simple closed curve having no intersections with the other two curves.  Hence, two of the weights $w_{1,2}$, $w_{1,3}$, $w_{2,3}$ are zero. Since at most one of $i,j,k$ is zero it follows that $\left<F_{i,j,k},D_{||}' \right>=0$. Thus, $\phi_{i,j,k}$ is invariant under the Reidemeister 2 move.
\newline
\newline
\underline{Reidemeister 3:} Since $\phi_{i,j,k}$ is invariant under all Reidemeister 1 and 2 moves, it is sufficient to consider only the Reidemeister move depicted in Figure \ref{omega3} and the Reidemeister move obtained from Figure \ref{omega3} by switching the three drawn crossings \cite{Ost}.
\newline
\newline
It follows by an argument similar to that given in the case of the Reidemeister 2 move that there is a one-to-one correspondence between subdiagrams $D_{||}'$ of the left hand side and subdiagrams $D_{||}''$ of the right hand side of the move when $D_{||}'$ contains zero or one of the arrows involved in the move. Therefore, it suffices to consider only the case that both arrows of $D_{||}'$ are involved in the move. The previous observations imply that the Reidemeister 3 move corresponds to the vanishing of $F_{i,j,k}$ on the two relations depicted in Figure \ref{fig_phi_ijk_omega3}. Checking all possibilities and noting that $i \ne j \ne k \ne i$, we see that $F_{i,j,k}$ vanishes on both relations.  Hence, $\phi_{i,j,k}$ is invariant under the Reidemeister 3 move.
\end{proof}

\begin{figure}[htb]
\[
\begin{array}{ccc}
\begin{array}{c} \psfig{figure=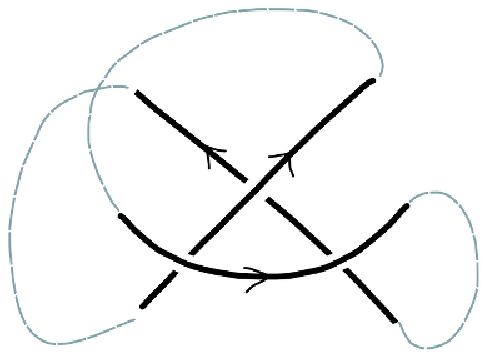} \end{array} & \leftrightharpoons & \begin{array}{c} \psfig{figure=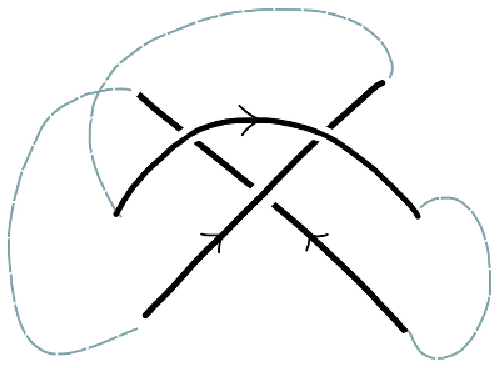} \end{array}
\end{array}
\]
\caption{A Reidemeister $3$ move.}\label{omega3}
\end{figure}

\begin{figure}
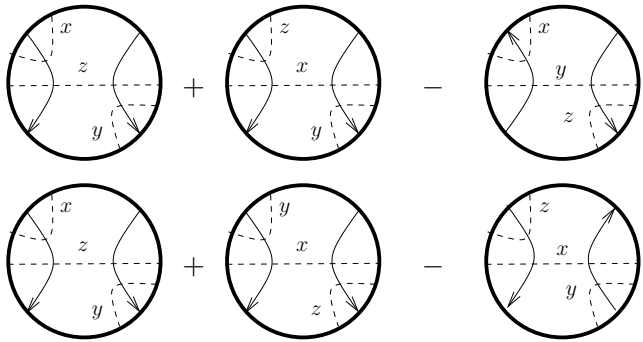

\begin{eqnarray*}
\begin{array}{c}
\scalebox{.4}{\input{phi_omega_rels_1_1.pstex_t}}
\end{array}+\begin{array}{c}
\scalebox{.4}{\input{phi_omega_rels_1_2.pstex_t}}
\end{array} &-& \begin{array}{c}
\scalebox{.4}{\input{phi_omega_rels_1_3.pstex_t}} \end{array} \\
\begin{array}{c}
\scalebox{.4}{\input{phi_omega_rels_1_1.pstex_t}}
\end{array}+\begin{array}{c}
\scalebox{.4}{\input{phi_omega_rels_2_2.pstex_t}}
\end{array}       &-&  \begin{array}{c}
\scalebox{.4}{\input{phi_omega_rels_2_3.pstex_t}} \end{array}\\ 
\end{eqnarray*}
\caption{Two formal sums of subdiagrams which appear in Reidemeister 3 moves.}\label{fig_phi_ijk_omega3}
\end{figure}

\subsection{Example} A Gauss diagram and a corresponding virtual knot diagram $K$ are shown in Figure \ref{fig:loopexample}. A straightforward computation verifies that $ \phi_{1,0,2} (K) = \phi_{1,2,0}(K)=1$. On the other hand, $\phi_{2,0,1}(K)=0$.

\begin{figure}[htb] 
\scalebox{0.5}{
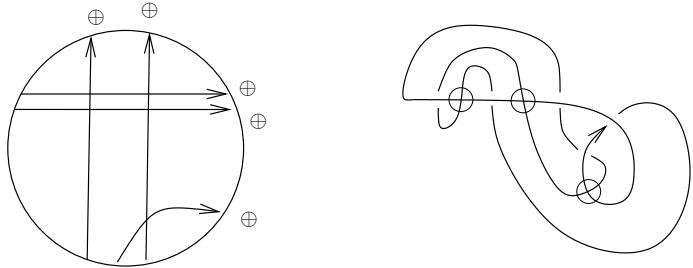}
\caption{The virtual knot $K$ satisfies $ \phi_{1,0,2}(K)=\phi_{1,2,0}(K)=1$.}
\label{fig:loopexample}
\end{figure}
 
\section{The Three Loop Framed Isotopy Invariant} \label{sec_phi}

\noindent In this section we introduce the three loop framed isotopy invariant $\Phi^{fr}$ of virtual knots. The invariant is valued in an abelian group $A$. We begin in Section \ref{sec_group_a} with a presentation of $A$. The presentation is useful for establishing the properties of $\Phi^{fr}$ (see Section \ref{sec_properties}).  However, the presentation is not always the most useful for the purpose of comparing values of the invariant.  We prove that $A$ is isomorphic to the direct sum of infinitely many copies of $\mathbb{Z}_2$. Using this isomorphism, it is easy to distinguish between the values of $\Phi^{fr}$ in the group $A$.
\newline
\newline
The function $\Phi^{fr}$ is proved to be a framed isotopy invariant in Section \ref{sec_phi}. A particularly interesting property of $\Phi^{fr}$ is that it can also be used as in invariant of knots.  More precisely, if $K_1$ and $K_2$ are equivalent virtual knot diagrams having the same writhe modulo 2, then $\Phi^{fr}(K_1)=\Phi^{fr}(K_2)$ (proved in Section \ref{sec_frame_virt_isotopy}). This makes $\Phi^{fr}$ useful in applications such as detecting non-invertible virtual knots and virtual knots which are not equivalent to their mirror images. This will be explored in more detail in Section \ref{sec_geometric}.

\subsection{The Group $A$} \label{sec_group_a} Let $A $ denote the group generated by the set of symbols $\{A_{ij}:i,j \in \mathbb{Z} \}$ and having the following relations:
\begin{align*}
A_{ii} &=1,  &  
A_{ij} ^2 &= 1, \\
A_{ij} A_{lm} &= A_{lm} A_{ij},  &
A_{ij}A_{jk} & = A_{ik}.
\end{align*}

\noindent The given presentation of $A$ is useful for establishing the properties of the three loop framed isotopy invariant. To aid in computation of the knot invariant, we will prove that $A$ is isomorphic to a direct sum of infinitely many copies of $\mathbb{Z}_2$. First, we need the following lemma.

\begin{lemma} For all $i,j \in \mathbb{Z}$, $A_{ij}=A_{ji}$.
\end{lemma}
\begin{proof} Since $A_{ij}^2=1$, we have that $A_{ij}^{-1}=A_{ij}$.  Moreover, we have that $A_{ij}A_{ji}=A_{ii}=1$.  Thus, $A_{ji}=A_{ij}^{-1}=A_{ij}$.
\end{proof}

\begin{theorem} There is a group isomorphism $\displaystyle{f:A \to \mathbb{Z}_2^{\infty}=\bigoplus_{i=-\infty}^{\infty} \mathbb{Z}_2}$.
\end{theorem}
\begin{proof} Enumerate the copies of $\mathbb{Z}_2$ in $\mathbb{Z}_2^{\infty}$ as $\{\ldots,-1,0,1,\ldots\}$. Let $\vec{0}$ denote the identity element of $\mathbb{Z}_2^{\infty}$. We define a function $f:A \to \mathbb{Z}_2^{\infty}$ as follows. By the previous lemma, it is sufficient to define $f$ on generators $A_{ij}$ where $i \le j$. Define $f(A_{ii})=\vec{0}$ for all $i\in \mathbb{Z}$.  For all $i\in \mathbb{Z}$, define $f(A_{i,i+1})$ to be the element which has a $1$ in the $i$-th copy of $\mathbb{Z}_2$ and a zero in all other copies of $\mathbb{Z}_2$. For $i<j-1$, define $f(A_{ij})$ to be the element which has a one in positions $i,i+1,i+2,\ldots,j-1$ and zeros in all other positions. If $i>j$, define $f(A_{ij})=f(A_{ji})$. This defines $f$ on generators.
\newline
\newline
\noindent To show that $f$ is well-defined, it is sufficient to show that $f(R)=\vec{0}$ for all defining relations $R$ of $A$. By definition $f(A_{ii})=\vec{0}$.  Also, $f(A_{ij}^2)=f(A_{ij})+f(A_{ij})=2(f(A_{ij}))=\vec{0}$. Finally, note that for $i<j<k$,  $f(A_{ij})+f(A_{jk})$ is a string of ones beginning in the $i$-th copy of $\mathbb{Z}_2$ and ending in the $(k-1)$-th copy of $\mathbb{Z}_2$. As this element is $f(A_{ik})$, we see that $f(A_{ij}A_{jk} A_{ik}^{-1})=\vec{0}$. If $j<i<k$, then $f(A_{ij})+f(A_{jk})=f(A_{ji})+f(A_{jk})$. Computing $f(A_{ji})+f(A_{jk})$ we see that we get zeros from $-\infty$ to $i-1$, ones from $i$ to $k-1$, and zeros from $k$ to $\infty$.  This is $f(A_{ik})$, so it follows again that $f(A_{ij}A_{jk} A_{ik}^{-1})=\vec{0}$. There are some other cases to consider and each follows from a similar argument. Thus, $f$ is well defined.
\newline
\newline
\noindent Next we show that $f$ is onto. An element $y$ of $\mathbb{Z}_2^{\infty}$ has only finitely many entries which are non-zero.  In particular, it can be broken into finitely many ``runs of ones''.  If there is a one going from the $i$-th to the $k$-th position and zero in the $i-1$ and $k+1$ position, it is the image of $A_{i,k+1}$. Since only finitely many of these are necessary to achieve $y$, it follows that $f$ is onto.
\newline
\newline
\noindent Lastly we must show that $f$ is one-to-one.  We will do this by constructing an inverse to $f$. Note that $\mathbb{Z}_2^{\infty}$ is generated by the elements $y_i$ having a 1 in the $i$-th copy of $\mathbb{Z}_2$ and zero elsewhere. In fact, every element of $\mathbb{Z}_2^{\infty}$ may be written uniquely as a linear combination of the $y_i$. Define $f^{-1}$ on generators by $f^{-1}(y_i)=A_{i,i+1}$. The map $f^{-1}$ is well-defined and $f\circ f^{-1}$ and $f^{-1}\circ f$ are both identity functions. Thus, $f$ is one-to-one. This completes the proof.   
\end{proof}

\begin{remark} Let $e_i$ denote the infinite vector with a $1$ in the $i^{th} $ position and zero elsewhere. Then $A_{i, i+1} $ maps to $e_i$. More specifically, $A_{0,1} $ maps to $e_0 $ and $A_{-1,0} $ maps to $e_{-1}$. Then $ A_{-2,1} $ maps to $ e_{-2} + e_{-1} + e_{0} $ and $A_{-2,3} A_{1,2} $ maps to $ e_{-2} + e_{-1} + e_{0} + e_{2}$. 
\end{remark}

\subsection{The Three Loop Framed Isotopy Invariant} In this section, we give the precise definition of the three loop framed isotopy invariant $\Phi^{fr}$ and prove that it is indeed an invariant of framed virtual knots.
\newline
\newline
\noindent Let $K$ be a virtual knot diagram. We associate an element of $A$, denoted $\Phi^{fr}(K) $ as follows. Given two distinct crossings $p$ and $q$ in $K$, let $K_{pq} $ denote the virtual link obtained by taking oriented smoothings at $p$ and $q$. Let $ \mathcal{P} (K) $ denote the set of unordered pairs of crossings such that $K_{pq} $ is a three component virtual link diagram (so that the corresponding arrows are non-intersecting in the Gauss diagram). In symbols, $\mathcal{P}(K) = \{  \{ p, q \} | K_{pq} \text{ is a
three component link}  \}$. The \emph{three loop framed isotopy invariant} is defined to be:
\begin{equation}
\Phi^{fr}(K)= \prod_{\lbrace p,q \rbrace \in \mathcal{P}(K)} A_{w(p) w(q) } .
\end{equation}

\begin{theorem} The map $ \Phi^{fr}: \mathbb{Z} [ \mathscr{K}^{fr}] \rightarrow A $  is an invariant of framed virtual knots. \end{theorem}

\begin{proof} It will be shown that $ \Phi^{fr} $ is invariant under Reidemeister 2 and 3 moves, and the non-classical extended Reidemeister moves. Each element of $ \mathcal{P}(K) $ contributes a factor
$ A_{w(p) w(q) } $ to the product $ \Phi^{fr}(K) $. Figure \ref{fig:rmoves} shows the left hand side (LHS) and right hand side (RHS) of the classical and non-classical extended Reidemeister moves. 
Let $K$ denote the diagram of a virtual knot representing the left hand side of a Reidemeister move and let $K'$ denote the diagram of a virtual knot representing the right hand side of a Reidemeister move. Recall Proposition \ref{weightprop} which states that 
Reidemeister moves may create or delete crossings, but the weight of the other crossings is not changed. Let $ \lbrace p , q \rbrace \in  \mathcal{P} (K) $.
\newline
\newline
\noindent\underline{Reidemeister 2:}  If
both $p$ and $q$ are not involved in the Reidemeister 2 move, then $ \lbrace p, q \rbrace $ is an element of $ \mathcal{P} (K) $ and $ \mathcal{P} (K') $.  If only $ p $ is in a Reidemeister 2 move on the LHS, then there exists a crossing $t$ with opposite sign (also in diagram $K$) such that $ w(p) = w(t) $ and $ \lbrace t, q \rbrace \in \mathcal{P} (K)$. Then both $ \lbrace p , q \rbrace $ and $ \lbrace t, q \rbrace $ contribute  $ A_{w(p) , w(q) } $ to the product $ \Phi^{fr} (K) $, resulting in a net contribution of $1$. The crossings 
$p$ and $t $ do not appear in the diagram $K'$. If both $p$ and $q$ appear in the Reidemeister 2 move on the LHS, then  their net contribution to $ \Phi^{fr}(K) $ is $A_{w(p),w(p)}=1$. These crossings do not appear in $K'$. 
\newline
\newline
\noindent\underline{Reidemeister 3:} If $ \lbrace p, q \rbrace $ is not involved in the Reidemeister 3 move, then  $A_{w(p) w(q)} $ is a factor of both 
$ \Phi^{fr}(K) $ and $ \Phi^{fr}(K')$.  If $p $ is involved in the LHS of a Reidemeister 3 move and $q$ is not then $ \lbrace p, q \rbrace $ contributes $ A_{w(p) w(q) }$ to $ \Phi^{fr}(K)$. On the RHS of a Reidemeister 3 move, there is a crossing $p'$ that corresponds to $p$
and $ w(p) = w(p')$. The pair $ \lbrace p' ,q  \rbrace $ is an element of $ \mathcal{P} (K') $ and contributes $ A_{w(p) w(q) }  $ 
to $ \Phi^{fr}(K') $.
\newline
\newline
\noindent
  Suppose both $ p $ and $q$ are involved in the LHS of a Reidemeister 3. Let  $r$ denote the third crossing in the Reidemeister 3 move on the LHS, and let $p', q' $ and $r'$ denote the corresponding crossings in the Reidemeister 3 move on the RHS. 
Note that $ w(p) = w(p')$, $w(q) = w(q')$ and $w(r) = w(r')$ and that if
 $ \lbrace i, j \rbrace \in \mathcal{P} (K) $ then $ \lbrace i', j' \rbrace \not\in 
\mathcal{P} (K')$ for  $i, j \in \lbrace p,q,r \rbrace $ (verify this by sketching and smoothing oriented Reidemeister 3 moves). We assume without loss of generality that
$ \lbrace  \lbrace p, q \rbrace , \lbrace p, r \rbrace, \lbrace q, r \rbrace  \rbrace  \in \mathcal{P} (K) $. Since $ A_{w(p)w(q) } A_{w(p) w(r)} A_{w(q) w(r)} = 1 $, the net contribution from both the LHS and  RHS of the Reidemeister 3 is $1$.
\newline
\newline
\noindent The non-classical extended Reidemeister moves do not change the weight of crossings and the sets $ \mathcal{P} (K) $ and $ \mathcal{P} (K') $  are in one-to-one correspondence.
\end{proof} 

\subsection{Framed Invariant and Virtual Isotopy} \label{sec_frame_virt_isotopy} Even though $\Phi^{fr}$ is an invariant of framed virtual isotopy, it can still be used to distinguish between virtual knots. If $K_1$ and $K_2$ are equivalent virtual knots having the same writhe modulo 2, then $\Phi^{fr}(K_1)=\Phi^{fr}(K_2)$.  The aim of this subsection is to prove this theorem. It will be used several times in the sequel. For example, this property turns out to be useful for distinguishing geometric symmetries of virtual knots (see Section \ref{sec_geometric}). 

\begin{theorem}\label{writheandframed} If $K_1$ and $K_2$ are equivalent virtual knot diagrams and $writhe(K_1)=writhe(K_2)$, then $\Phi^{fr}(K_1)=\Phi^{fr}(K_2)$.
\end{theorem}
\begin{proof} Suppose that we have a sequence $R_1$, $R_2$,$\ldots$, $R_N$ of extended Reidemeister moves taking $K_1$ to $K_2$. $\Phi^{fr}$ is invariant under all moves except the Reidemeister 1 move. If none of the $R_i$, $1 \le i \le N$ is a Reidemeister 1 move, then the theorem is true.
\newline
\newline
\noindent Let $R_i$ be move with smallest subscript which is a Reidemeister $1$ move. Then $R_i$ either adds an isolated arrow to the Gauss diagram or deletes an isolated arrow from a Gauss diagram.  Suppose that $R_i$ adds an isolated arrow.  Then the writhe of the diagram has changed by $\pm 1$. We add a second isolated arrow to the diagram and mark it with an $e$ (the $e$ stands for $e$xtra). Since $A$ is abelian and $A_{ij}^2=1$, we must have that this operation does not change the value of $\Phi^{fr}$. On the other hand, suppose that $R_i$ deletes an arrow.  Then in this case we delete the arrow and add an isolated arrow marked $e$ in such a way that it does not interfere with the Reidemeister move $R_{i+1}$. Again, the value of $\Phi^{fr}$ has not changed.
\newline
\newline
\noindent Now we continue with the other Reidemeister moves.  If an arrow marked $e$ is ever in a position to interfere with any subsequent Reidemeister move $R_j$, we simply delete it an add it somewhere so that it no longer interferes. Since $writhe(K_1)=writhe(K_2)$, there must be other Reidemeister moves somewhere in the sequence from $K_1$ to $K_2$. When we have arrived at that point in the sequence, there are two possibilities: either there are no arrows marked $e$ or there is at least one arrow marked $e$. If there are no arrows marked $e$, then we repeat exactly the procedure for $R_i$ (first Reidemeister 1 move in the given sequence).
\newline
\newline
\noindent Now suppose we are at a Reidemeister 1 move in the sequence and that there is an arrow marked $e$. If the move increases the number of arrows, we add the arrow as defined in the move and delete any arrow marked $e$. If the move decreases the number of arrows, we delete the the arrow as defined by the move and delete any arrow marked $e$. In both cases, the value of $\Phi^{fr}$ does not change.
\newline
\newline
\noindent By this procedure, we see that the Gauss diagram has at most one arrow marked $e$ at any given position in the modified sequence. Since $writhe(K_1)=writhe(K_2)$, we must have that the final diagram has no arrows marked $e$. Now the last diagram in the sequence is a diagram $K_2$. Since the value of $\Phi$ has remained unchanged, we conclude that $\Phi^{fr}(K_1)=\Phi^{fr}(K_2)$.
\end{proof}

\begin{corollary}\label{mod2frame} Let $K_1$ and $K_2$ be equivalent virtual knot diagrams.
\begin{enumerate}
\item If $\emph{writhe}(K_1)\equiv \emph{writhe}(K_2) \pmod{2}$, then $\Phi^{fr}(K_1)=\Phi^{fr}(K_2)$.
\item If $ \emph{writhe}(K_1) \not\equiv \emph{writhe} (K_2) \pmod{2} $ and $ \emph{writhe} (K_1) $ is odd, then
\newline
$ \Phi^{fr}(K_1) =  \Phi^{fr} (K_2) \cdot  \underset{\text{crossings } a \text{ in } K_2}{ \prod}  A_{0, w(a)} $.
\end{enumerate} 
\end{corollary}
\begin{proof} If $\emph{writhe}(K_1)=\emph{writhe}(K_2)$, the result follows from the previous theorem. Otherwise, we add an even number of isolated arrows to a Gauss diagram of $K_2$ to obtain a virtual knot diagram $K_2'$ so that $\emph{writhe}(K_1)=\emph{writhe}(K_2')$. Certainly, $K_1$ and $K_2'$ are equivalent virtual knot diagrams. Hence, $\Phi^{fr}(K_1)=\Phi^{fr}(K_2')$. Now, since $A_{ij}^2=1$ for all $i,j \in \mathbb{Z}$ and $A_{00}=1$, it follows that the added arrows do not change the value of $\Phi^{fr}$. Hence, $\Phi^{fr}(K_1)=\Phi^{fr}(K_2')=\Phi^{fr}(K_2)$.
\newline
\newline
\noindent If $ \emph{writhe}(K_1) $ is odd, we add an odd number of isolated arrows to $K_2$ to obtain a virtual knot diagram $K_2'$ so that $  \emph{writhe}(K_1) = \emph{writhe} (K_2') $. The diagrams $K_1 $ and $K_2'$ are equivalent as virtual knot diagrams and $ \Phi^{fr} (K_2') = \Phi^{fr}(K_2) \prod 
 A_{0,w(a)}$. 
\end{proof}

\subsection{Examples}
The framed virtual isotopy invariant $ \Phi^{fr}$ does not detect classical knots.  There are non-classical virtual knots that are not detected by this invariant. Notably, the virtual trefoil is not detected. 

\begin{figure}[htb] \scalebox{0.3}{
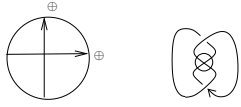}
\caption{A virtual trefoil and Gauss diagram.}
\label{fig:vtref}
\end{figure}
\hspace{1cm}
\newline
\newline
\noindent In Figure \ref{fig:mini}, we have a schematic diagram of a family of knots $ O(n)$ with an even number of arrows. In Figure \ref{fig:mini}, all arrows are positively signed and the dotted arrows represent a set of  $n$ ``stacked'' arrows.  The stacked arrows are parallel, co-oriented and positively signed. If two arrows $c_i $ and $c_j$ are in a set of stacked arrows, then $c_i $ and $c_j$ are parallel, co-oriented and $N_{c_i} = N_{c_j} $. As a result, the stacked arrows have the same weight and
 $ w(c_i) = w(c_j) $. For this reason, we suppress the individual labels on the arrows and by $w(c) $, we mean the weight of any arrow in the stack labeled $c$. 

\begin{proposition} There is a family of oriented knots corresponding to the family of Gauss diagrams $ O(n) $.  
\begin{equation}
\Phi^{fr}(O(n) ) =\begin{cases}  A_{0, n} & \text{ if $n$ is even } \\
A_{0, -2}  & \text{ if $n$ is odd}  \end{cases} 
\end{equation}
\end{proposition} 

\begin{proof} We compute the weights of each individual arrow and the stacked arrows.

\begin{align*}
w(a) &= n, & w(b) &=0, \\
w(c) &=-2 ,& w(d)&=1.
\end{align*}

\noindent Using multi-set notation, we find that
 $ \mathcal{P} (O(n)) =  \lbrace \lbrace a,b \rbrace, n \cdot \lbrace a,d \rbrace, n^2 \cdot \lbrace c,d \rbrace \rbrace $. Note that if $n$ is even, $ \Phi^{fr}(K) = A_{0,n} $. If $n$ is odd, $\Phi^{fr}(O(n)) = A_{0, -2} $. Therefore, we obtain the values $ A_{0,2k} $ for $k \in 
\mathbb{Z}$.
\end{proof}

\begin{figure}[htb] \scalebox{0.5}{
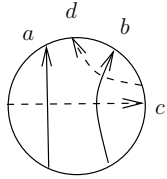}
\caption{Definition of the virtual knot $O(n)$.}
\label{fig:mini}
\end{figure}

\noindent A natural question to consider is if $ \Phi^{fr}: \mathbb{Z} [ \mathscr{K}^{fr} ] \rightarrow A $ is a surjective mapping (see \cite{cheng_knot} for a related result for the Cheng polynomial).  The following proposition, whose proof is omitted, shows that $\Phi^{fr}$ is not surjective. 

\begin{proposition} For $K \in \mathscr{K}^{fr} $, $ f \circ \Phi^{fr}(K) $  has an even number of non-zero entries in $ \mathbb{Z}_2 ^{ \infty}$. 
\end{proposition} 

\begin{figure}[htb] \epsfysize = 1 in
\centerline{\epsffile{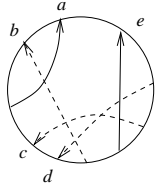}}
\caption{An example of a virtual knot $K$ such that $\Phi^{fr}(K)$ can be represented with only odd subscripts.}
\label{fig:prime}
\end{figure}

\noindent Lastly, consider the Gauss diagram in Figure \ref{fig:prime}. This is an example in which the value of $\Phi^{fr}$ is written so that it contains only odd subscripts. All arrows are positive and the dotted arrows represent a set of $3$ stacked arrows (which consequently have the same weight). The weights of the arrows and stacked arrows are:

\begin{align*}
w(a)&= -3, & w (b) &= -5,  &
w(c) &= 1,  \\  w(d) &= 7, &
w(e) &=-6. 
\end{align*}

\noindent It follows that $ \Phi^{fr}(K) = A_{1,7} A_{-3,-5}$.

\section{Properties of $\phi_{i,j,k}$ and $\Phi^{fr}$} \label{sec_properties}

\subsection{Finite-Type Invariants} \label{sec_fti}
We refer the reader to \cite{gpv} for a more complete discussion of finite-type invariants (or Vassiliev invariants). We will be using the definition of finite-type invariant for virtual knots originally given in \cite{kvirt}. 
\newline
\newline
\noindent A singular virtual knot diagram is a virtual knot diagram with singular crossings that are indicated by a vertex. Singular virtual knots are equivalence classes determined by the Reidemeister moves, virtual Reidemeister moves and the singular crossing moves (Figure \ref{fig:singular}). 

\begin{figure}[htb] 
\epsfysize = 1 in
\epsffile{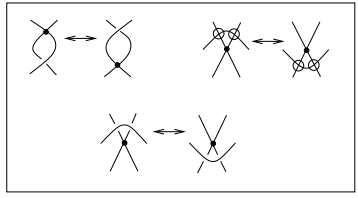} 
\caption{The singular moves.}
\label{fig:singular}
\end{figure}
\hspace{1cm}
\newline
\newline
\noindent Let $K$ be a diagram with $n+1 $ singular crossings. Let $ \sigma $ denote a $(n+1)$-tuple of 1's and 0's. Denote by $ |\sigma| $ the the number of nonzero entries of $\sigma$. The diagram $K_{\sigma } $ is obtained by replacing the $i^{th} $ singular crossing with a negative (respectively positive) crossing if a $1$ (respectively $0$) appears in the $i^{th} $ position of $ \sigma$. Then a virtual knot invariant $v$ is said to be a \emph{finite-type invariant of degree} $ \leq n$ if for all diagrams $K$ having $n+1$ singular crossings we have:

\begin{equation} \label{finite}
\sum_{ \sigma}  (-1)^{ | \sigma|}  v (K_{\sigma}) = 0.
\end{equation}

\subsubsection{Three Loop Invariant} It is proved that the three loop isotopy invariant is a finite type invariant of degree $\le 2$ but not $\le 1$.
 
\begin{theorem} The invariant $ \phi_{i,j,k} $ is a finite-type invariant of degree two.  \end{theorem}

\begin{proof} Let $K_{a,b,c} $ be a virtual knot diagram with distinct singular crossings $a,b,$ and $c $. Let $ \sigma $ and $ \omega $  be elements of $ \mathbb{Z}_2 ^3 $ so that $K_{ \sigma} $ and 
$ K_{\omega} $ denote resolutions of the singular virtual knot. We let $G_{\sigma} $ denote the Gauss diagram that corresponds to $ K_{ \sigma} $.  
We will use subscripts to identify specific Gauss diagrams. Let $l, m$ and $n$ be elements of $ \mathbb{Z}_2 $, the notation $G_{lmn} $ identifies the diagram with resolution $l,m$, and $n$ of the arrows $a,b$, and $c$ respectively. For example:
$G_{000} $ denotes the diagram with all three singular crossings positively resolved, 
$ G_{0mn} $ denotes a diagram where the arrow $a$ is positively signed and the notation
$G_{10n} $ denotes a diagram where $a$ is negatively signed and $b$ is positively signed.
We will denote a subdiagram of $ G_{ \sigma }$ (respectively $ G_{lmn} $) with two non-intersecting arrows as
 $D_{||} ^{ \sigma} $ (respectively $D_{||} ^{lmn} $). 
\newline
\newline
\noindent
The diagram $ G_{ \sigma} $ can be transformed into diagram $ G_{\omega} $ by changing the sign and direction of at most three arrows in $G_{ \sigma} $. Hence, for each subdiagram
$D_{||} ^{\sigma} $ in $G_{ \sigma} $ there is a corresponding subdiagram 
$D_{||} ^{\omega} $ in $ G_{ \omega} $. 
We say that $ D_{||} ^{ \sigma } $  \emph{corresponds} to $ D_{||}^{000} $ if we can obtain
$D_{||}^{ \sigma} $ from $D_{||} ^{000} $ by switching the orientation and sign of a (possibly empty) subset of the arrows corresponding to the singular crossings $a,b$, and $c$. 
For each $ \sigma \in 
\mathbb{Z}_2 ^3 $, we can obtain a subdiagram $D_{||} ^{ \sigma} $ that corresponds to 
$D_{||} ^{000}$.
We define the \emph{set of corresponding diagrams} to be
\begin{equation*} 
 \lbrace D_{||}^{ \sigma } |  \sigma  \in \mathbb{Z}_2 ^3  \text{ and }
D_{||} ^{ \sigma }  \text{ corresponds to } D_{||}^{000} \rbrace . 
\end{equation*}
\newline
\newline
\noindent
We compute $\phi_{ijk}(K_{abc}) $ by summing over all subdiagrams $ D_{||}'$ of $G_{\sigma} $.
\begin{align*} 
\phi_{ijk} (K_{abc}) &= \sum_{\sigma} (-1) ^{| \sigma|} \phi_{ijk} (K_{\sigma}) \\
&= \sum_{\sigma} (-1) ^{| \sigma|} \sum_{ D_{||} ' \subset G_{ \sigma} }
\langle F_{ijk} , D_{||}' \rangle  \\
&= \sum_{D_{||} ^{000} \subset G_{000} } \sum_{ \sigma} (-1)^{ | \sigma |} 
\langle F_{ijk}, D_{||} ^{ \sigma}  \rangle .
\end{align*}
It is sufficient to show that: 
\begin{equation} \label{zero}
\sum_{ \sigma} (-1)^{ | \sigma |} 
\langle F_{ijk}, D_{||} ^{ \sigma}  \rangle  =0.
\end{equation}
for every set of corresponding subdiagrams.
There are three cases to consider: arrows $a,b,$ and $c$ are not in $D_{||}^{000} $, exactly one of the arrows $a,b,c$ is contained in $ D_{||}^{000} $, or two of the arrows $a,b,c$ are in $ D_{||}^{000}$.
\newline
\newline
\noindent 
Suppose that the arrows corresponding to $a,b$ and $c$ are not in $D_{||} ^{000} $. Changing the sign and orientation of these arrows does not affect the weight or sign of the arrows in $D_{||} ^{000} $.  Hence, $ \langle F_{ijk}, D_{||} ^{ \sigma } \rangle =c $ for all $ \sigma $ in
$ \mathbb{Z}_2 ^3 $ where $ c \in \lbrace -1, 0 ,1  \rbrace $.  Then for this family of diagrams,
equation \ref{zero} is true.
\newline
\newline
\noindent Now suppose that $ D_{||} ^{000} $ contains one of $a,b,c$. Without loss of generality, we may suppose it is $a$. Changing the orientation and sign of the arrow $a$ affects the evaluation of $ \langle F_{ijk}, D_{||} ^{ \sigma} \rangle $. 
We note that 
\begin{align*}
 \langle F_{ijk}, D_{||} ^{0mn} \rangle &= c  \text{ for }m,n  \in \mathbb{Z}_2 , 
 c \in \lbrace -1, 0 , 1 \rbrace, \\
\langle F_{ijk}, D_{||} ^{1mn} \rangle &=b \text{ for } m,n \in \mathbb{Z}_2, b
\in \lbrace -1, 0, 1 \rbrace.
\end{align*} 
Summing these terms with appropriate sign, we see that equation \ref{zero} is verified.
\newline
\newline
Suppose that $D_{||} ^{000} $ contains $a$ and $b$ without loss of generality. 
Again, we note that changing the sign and direction of these arrows affects the evaluation. However changing the sign and direction of $c$ does not affect the evaluation. 
For $m \in \mathbb{Z}_2 $, we have that
\begin{align*}
\langle F_{ijk} D_{||} ^{00m} \rangle &= c_{00},  & \langle F_{ijk} D_{||} ^{01m} \rangle &= c_{01}, \\
\langle F_{ijk} D_{||} ^{10m} \rangle &= c_{10}, & \langle F_{ijk} D_{||} ^{11m} \rangle &= c_{11}, 
\end{align*}
where $ c_{00}, c_{01}, c_{10}, $ and $c_{11} $ are all elements of $ \lbrace -1,0,1 \rbrace $. 
The sum of these terms with appropriate sign is zero, again verifying equation \ref{zero}. 
\newline 
\newline 
We have shown that the invariant has degree $\leq 2 $.  We now show that $ \phi_{ijk} $ 
does not vanish on a knot with two singular crossings. 
We compute the  ``second derivative'' of $\phi_{ijk}$ and show there is a case when
$ \phi_{ijk}$  does not vanish. 
Let $K_{ab} $ denote a singular virtual knot diagram with two singular crossings, and let $K_{ \sigma}$
with $ \sigma \in \mathbb{Z}_2 ^2  $ denote a resolution of $K_{ab}$. 
Let $G_{ \sigma} $ denote the Gauss diagram corresponding to $K_{\sigma} $.   
The Gauss diagram $G_{00} $ of $K_{00} $ is shown in Figure \ref{fig:degree1}.
In $G_{00} $, the positive arrows $a$ and $b$ correspond to resolutions of the singular crossings $a$ and $b$. The dashed arrows indicate stacked arrows (positively signed, parallel, co-oriented arrows) and their labels indicate the number of arrows in the stack. 
\begin{figure}[htb] 
\scalebox{0.5}{
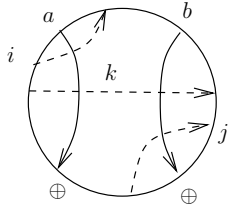 }
\caption{A Gauss diagram of $K_{00}$.}
\label{fig:degree1}
\end{figure}
\hspace{1cm}
\newline
\newline
\noindent 
As before, for each subdiagram $D_{||} ^{00} \subset G_{00} $ there is a set of corresponding subdiagrams: $D_{||}^{00}, D_{||}^{01}, D_{||}^{10}$, and $D_{||}^{11}$. 
We have that
\begin{align*}
\phi_{ijk} (K_{ab}) & =  \sum_{ \sigma}  (-1) ^{| \sigma|} \phi_{ijk} (K_{\sigma})  \\
&= \sum_{\sigma} (-1) ^{| \sigma|} \sum_{ D_{||}' \subset G_{ \sigma}} \langle 
F_{ijk}, D_{||}' \rangle  \\
&= \sum_{ D_{||}^{00} \subset G_{00} } \sum_{\sigma} (-1) ^{| \sigma |} \langle 
F_{ijk}, D_{||} ^{ \sigma } \rangle.
\end{align*}    
\newline
\newline
\noindent
We will show that $ \phi_{ijk} (K_{ab}) \neq 0 $ by considering families of corresponding subdiagrams.
If we consider families where if $D_{||} ^{\sigma} $ contains zero or one arrow obtained by resolving a singularity, then $ \phi_{ijk} $ applied to the signed sum of the family of diagrams is zero. 
There is a unique family of corresponding diagrams where both arrows in the subdiagram are resolutions of the singular points. In this case,
\begin{equation*}
\langle F_{ijk}, D_{||} ^{00} \rangle - \langle F_{ijk}, D_{||} ^{01} \rangle -\langle F_{ijk}, D_{||} ^{10} \rangle + \langle F_{ijk}, D_{||} ^{11} \rangle \neq 0. 
\end{equation*}
\newline
\newline
\noindent
The subdiagram $D_{||}^{00} $ is  shown below and the structure of the other subdiagrams $D_{||} ^{\sigma} $ in the set of corresponding diagrams. 
\[
\begin{array}{cc}
D_{00}=\begin{array}{c}\scalebox{.5}{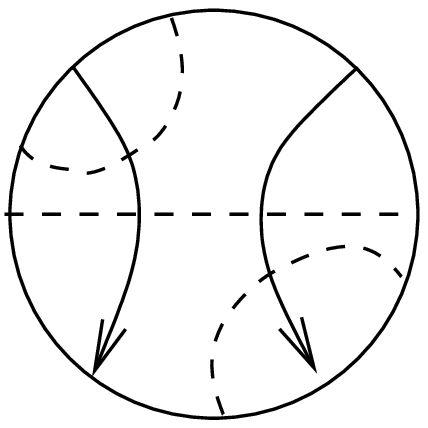} \end{array}, & D_{01}=\begin{array}{c}\scalebox{.5}{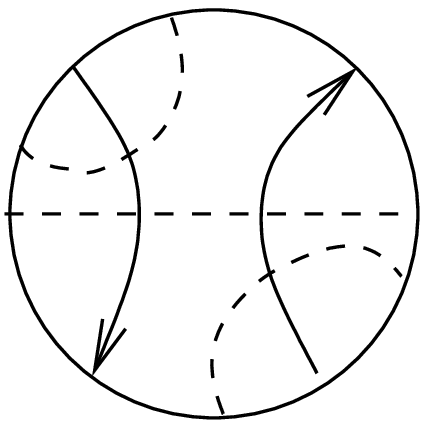}\end{array}, \\
D_{10}=\begin{array}{c}\scalebox{.5}{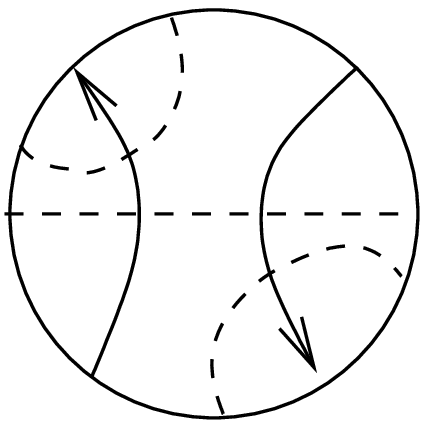} \end{array}, & 
D_{11}=\begin{array}{c}\scalebox{.5}{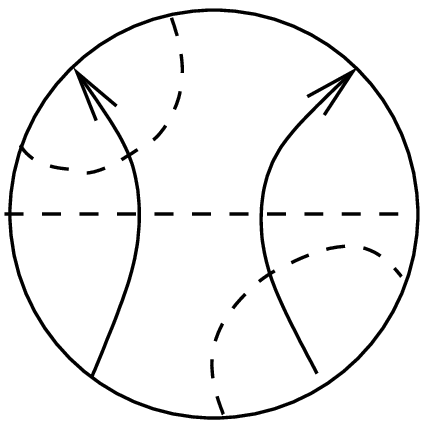} \end{array}. \\
\end{array}
\]
Evaluating the family of diagrams gives: 
\begin{align*}
\langle F_{ijk} , D_{00} \rangle &=1,  & \langle F_{ijk} ,  D_{01} \rangle &=0,  \\
\langle F_{ijk} , D_{10} \rangle &=0,  & \langle F_{ijk} , D_{11} \rangle &=0.
\end{align*}
Then $ \phi_{ijk} (K_{ab}) = 1 $ and $\phi_{i,j,k}$ is not a finite-type invariant of degree $ \le 1$.
\end{proof}

\subsubsection{Three Loop Framed Invariant} We prove in this section that the three loop framed isotopy invariant $ \Phi^{fr} $ is a degree one finite-type invariant.  

\begin{remark} We use a summation symbol in equation \ref{finite}. The group $A$ is  multiplicative and every element is its own inverse. As a result, equation \ref{finite} is written: 

\begin{equation}
\prod_{ \sigma}   v (K_{\sigma}) = 1.
\end{equation}
\end{remark}

\begin{theorem}
The invariant $ \Phi^{fr}(K) $ is finite type invariant of degree one.
\end{theorem}

\begin{proof} We consider a singular virtual knot diagram $K_{ab} $ with distinct singular crossings $a$ and $b$. Let $ \sigma \in \mathbb{Z}_2 ^2 $.  Consider $ \lbrace p, q \rbrace  \in 
\mathcal{P} (K_{00}) $, we say that $ \lbrace p^{ \sigma} , q^{ \sigma}  \rbrace
\in \mathcal{P} (K_{ \sigma} ) $ corresponds to  $ \lbrace p, q \rbrace  \in 
\mathcal{P} (K) $ if $K_{p^{ \sigma} q^{\sigma}} $ can be obtained from $ K_{pq} $ by switching a subset of the resolutions of the singular crossings.  Hence, for each $  \lbrace p, q \rbrace \in 
\mathcal{P} (K_{00}) $ we obtain a set of four corresponding crossing pairs. 
 We have that
\begin{align*} 
\Phi^{fr} (K_{ab}) &= \prod_{ \sigma} \Phi ^{fr} (K_{ \sigma}) \\
&= \prod_{ \sigma}  \prod_{ \lbrace p, q \rbrace \in \mathcal{P} (K_{ \sigma})} A_{w(p),w(q)} \\
&= \prod_{ \lbrace p, q \rbrace \in \mathcal{P} (K_{00}) } \prod_{ \sigma} A_{w(p^{ \sigma}),
w(q^{ \sigma})}.
\end{align*}
We will show that:
\begin{equation*}
\prod_{ \sigma} A_{w(p^{ \sigma}),
w(q^{ \sigma})} =1.
\end{equation*}
That is, the product obtained from a set of four corresponding crossing pairs is one.
\newline
\newline
\noindent Suppose that the set $ \lbrace p, q \rbrace \in \mathcal{P} (K_{00}) $ does not contain a resolution of $ a$ or $b$. Then $w(p) = w( p^{ \sigma}) $ and $ w(q) = w( q^{ \sigma}) $. 
Hence,
\begin{equation*}
\prod_{ \sigma} A_{w(p^{ \sigma}),
w(q^{ \sigma})} =(A_{w(p), w(q)} )^4 =1.
\end{equation*}
\newline
\newline
\noindent Suppose that the set $ \lbrace p, q \rbrace \in \mathcal{P} (K_{00}) $ contains the resolution of $ a$. If $p$ is a resolution of $a$, then $w(p) = w( p^{01}) $,
$w(p) = -w(p^{10}) = -w(p^{11})$ and $ w(q) = w( q^{ \sigma}) $ for all $ \sigma \in \mathbb{Z}_2 ^2$. Now:
\begin{equation*}
\prod_{ \sigma} A_{w(p^{ \sigma}),
w(q^{ \sigma})} =(A_{w(p), w(q)} )^2  (A_{-w(p), w(q)} )^2=1.
\end{equation*}
\newline
\newline
\noindent Suppose that the set $ \lbrace p, q \rbrace \in \mathcal{P} (K_{00}) $ contains the resolution of $ a$ and $b$. Let  $p$ (respectively $q$) be the resolution of $a$ (respectively $b$). Then 
\begin{gather*}
w(p) = w( p^{01}) = -w(p^{10}) = -w(p^{11}) \\
w(q) = w(q^{10}) = -w(q^{01}) = -w(q^{11}) .
\end{gather*}
Computation shows that
\begin{equation*}
\prod_{ \sigma} A_{w(p^{ \sigma}),
w(q^{ \sigma})} =1.
\end{equation*}
\newline
\newline
\noindent
We show that $ \Phi^{fr} $ is not degree zero. The Gauss diagrams,
$G_0 $ and $G_1 $, shown in Figure \ref{fig:phidegree1} correspond to the two resolutions of a singular knot diagram $K_a$.  In the Gauss diagrams, the dashed arrow represents a stack of
$i$ parallel, co-oriented, positive arrows and the arrow labeled $a$ corresponds to the singular crossing.
\begin{figure}[htb] 
\scalebox{0.5}{
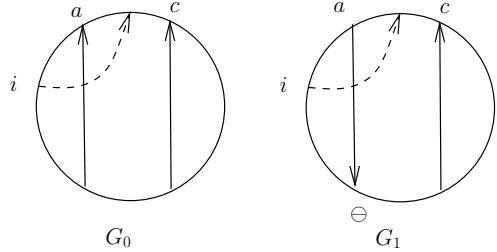 }
\caption{The Gauss diagrams $G_0 $ and $G_1$.}
\label{fig:phidegree1}
\end{figure}
\hspace{1cm}
\newline
\newline
\noindent
By definition, we have that:
\begin{align*}
\Phi^{fr} (K_a) = \Phi^{fr} (K_0)  \Phi^{fr} (K_1).
\end{align*} 
 If $i$ is odd then computation shows that $ \Phi^{fr} (K_0) = A_{-1, i} $ and 
$ \Phi^{fr} (K_1) = A_{-i,-1} $. Therefore
$\Phi^{fr} (K_a) = A_{-i, i} $
\noindent and $ \Phi ^{fr} $ is a degree one finite type invariant.
\end{proof} 

\subsection{Connected Sum} \label{sec_connectsum} We consider the connected sum of two diagrams. Let $K_1$ and $K_2 $ be virtual knot diagrams and let $K_1 \sharp K_2 $ denote a connected sum. A connected sum of two virtual knot diagrams is formed by selecting points on each diagram and then splicing the diagrams together at those points. The knot obtained from this operation depends on the selected points. 
\newline
\newline
As Gauss diagrams, $D_1 \sharp D_2 $ is formed by splicing the two Gauss diagrams at the point on the circle corresponding to the selected points on $K_1$ and $K_2$. 

\subsubsection{Three Loop Invariant}

\begin{theorem} If $ i,j,$ and $k$ are all non-zero then $ \phi_{ijk} (K_1 \sharp K_2 ) 
= \phi_{ijk} (K_1) + \phi_{ijk} (K_2) $. 
\end{theorem}

\begin{proof} Let $D_i $ be the Gauss diagram corresponding to $K_i$. Let $a_i$ be an arrow in $D_i$ and let $D_{a_1,a_2}$ denote the subdiagram of $D_1 \sharp D_2$ containing $a_1$ and $a_2$. Since $a_1$ and $a_2$ are non-intersecting, $D_{a_1,a_2}$ is ``counted'' by $\phi_{i,j,k}$. However, there are no arrows of $D_1 \sharp D_2$ which cross both $a_1$ and $a_2$. Hence $w_{13}=0$ and the result follows.
\end{proof} 

\subsubsection{Three Loop Framed Invariant}
\begin{theorem} \label{connect} Let $K_1 $ be a Gauss diagram with arrows $ \lbrace v_1, v_2, \ldots v_n \rbrace $ with corresponding weights $ p_1, p_2, \ldots p_n$. Let $K_2 $ be a Gauss diagram with arrows $ \lbrace w_1, w_2, \ldots w_m \rbrace $ and corresponding weights
$ q_1, q_2, \ldots q_m $. For $n$ even, let $E =  A_{p_1 p_2} A_{p_3 p_4} \ldots A_{p_{n-1} p_n}$. For both $n$ and $m$ odd, let $ O_1 = A_{p_1 p_2} A_{p_3 p_4} \ldots A_{p_{n-2} p_{n-1}}$ and 
$O_2 = A_{q_1 q_2} A_{q_3 q_4} \ldots A_{q_{m-2} q_{m-1}}$. Then:
\begin{equation}
\frac{\Phi ^{fr} (K_1 \sharp K_2)}{\Phi ^{fr}(K_1) \Phi ^{fr}(K_2)}  = \begin{cases} 1  &  \text{if } m,n \text{even} \\ E_1  & \text{if } n \text{ even, } m \text{ odd} \\
O_1 O_2   A_{p_n q_m} & \text{if } m,n \text{ odd} .  \end{cases}
\end{equation}
\end{theorem}
\begin{proof} Suppose $n=2k$. For all $i,j$, the arrows $v_i $ and $w_j$ do not intersect. As a result, the pairs contribute to the invariant. We compute the product:
\begin{equation}
\begin{matrix} A_{p_1 q_1} & A_{p_1 q_2} & \ldots & A_{p_1 q_m} \\
 A_{p_2 q_1} & A_{p_2 q_2} & \ldots & A_{p_2 q_m} \\
\vdots  &  \ddots  &      & \vdots \\
 A_{p_{2k-1} q_1} & A_{p_{2k-1} q_2} & \ldots  & A_{p_{2k-1} q_m} \\
 A_{p_{2k} q_1} & A_{p_{2k} q_2} & \ldots &  A_{p_{2k} q_m} 
\end{matrix}
\end{equation}
\noindent This reduces to the product:
\begin{equation}
(A_{p_1 p_2})^m (A_{p_3 p_4})^m \ldots (A_{p_{2k-1} p_{2k}})^m.
\end{equation}
Now we compute the ratio:
\begin{equation}
\frac{\Phi (K_1 \sharp K_2)}{\Phi(K_1) \Phi(K_2)}  = \begin{cases} 1  &  \text{ if } m \text{ even} \\
     E  & \text{ if }  m \text{ odd} 
\end{cases}
\end{equation}
\noindent If $m$ and $n$ are odd, let $m=2j+1$ and let $n=2k+1$:
\begin{equation}
\begin{matrix} A_{p_1 q_1} & A_{p_1 q_2} & \ldots & A_{p_1 q_{2j}} & A_{p_1 q_{2j+1} } \\
 A_{p_2 q_1} & A_{p_2 q_2} & \ldots & A_{p_2 q_{2j}} & A_{p_2 q_{2j+1}}   \\
\vdots  &  \ddots  &    & \vdots  & \vdots \\
 A_{p_{2k-1} q_1} & A_{p_{2k-1} q_2} & \ldots & A_{p_{2k-1} q_{2j}} & A_{p_{2k-1} q_{2j+1}} \\
 A_{p_{2k} q_1} & A_{p_{2k} q_2} & \ldots & A_{p_{2k} q_{2j} } & A_{p_{2k} q_{2j+1} }\\
 A_{p_{2k +1} q_1} & A_{p_{2k+1} q_2} & \ldots & A_{p_{2k+1} q_{2j}} & A_{p_{2k+1} q_{2j+1}}  \\
\end{matrix}
\end{equation}
\noindent This product reduces to:
\begin{equation} 
(A_{p_1 p_2} A_{p_3 p_4} \ldots A_{p_{n-2} p_{n-1}}) ( A_{q_1 q_2} A_{q_3 q_4} \ldots 
A_{ q_{m-2} q_{m-1}} ) A_{ p_n q_m},
\end{equation}
which is equal to $O_1 O_2 A_{p_n q_m}. $
\end{proof}

\begin{corollary} Let  $K_1 $ and $K_2 $  be virtual knot diagrams.
\begin{enumerate}
\item If $K_1 $ and $K_2 $ have even writhe then $ \Phi ^{fr} (K_1 \sharp K_2 ) = 
\Phi ^{fr} (K_1) \Phi ^{fr}(K_2) $. 
\item  If $ \tilde{K}_1 $ and $\tilde{K}_2 $ have odd writhe and $\tilde{K}_i $ is equivalent to 
a diagram $K_i$ with even writhe then $ \Phi ^{fr}(\tilde{K}_1 \sharp \tilde{K}_2) =  \Phi ^{fr} (K_1) 
\Phi ^{fr}( K_2) $.
\end{enumerate}  \end{corollary}
\begin{proof} This follows immediately from Theorem \ref{connect} and Corollary \ref{mod2frame}. \end{proof}

\subsection{Geometric Symmetries} \label{sec_geometric} Let $K$ be a oriented virtual knot diagram. We consider the behavior of $\phi_{i,j,k}$ and $\Phi^{fr}$ under various symmetries. If $K$ is an oriented virtual knot diagram, let $K^{-1}$ denote the virtual knot diagram obtained by reversing the orientation of $K$.  As in classical knot theory, $K^{-1}$ is called the \emph{inverse} of $K$. If $K \leftrightarrow K^{-1}$, then $K$ is said to be \emph{invertible}.
\newline
\newline
\noindent Let $K:S^1 \to \mathbb{R}^2$ be an oriented virtual knot diagram. Let $l$ be a line in $\mathbb{R}^2$ not intersecting $l$. Let $\rho:\mathbb{R}^2\to \mathbb{R}^2$ be the reflection about $l$. The \emph{mirror image} of $K$ is the oriented knot diagram $\rho \circ K$. The mirror image of $K$ will be denoted $\overline{K}$ (see Figure \ref{fig:mirrors}).
\newline
\newline
\noindent The \emph{switch} of $K$, denoted $K^s$, is defined to  be the knot obtained from $K$ by performing a crossing change at every crossing in $K$ (see Figure \ref{fig:mirrors}).  This is also called the \emph{horizontal mirror} of $K$. 

\begin{figure}[htb] 
\scalebox{0.3}{ 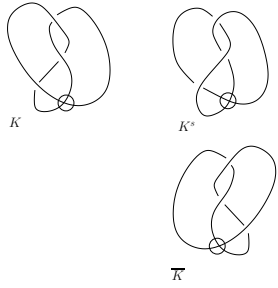}
\caption{Geometric Symmetries of $K$ }
\label{fig:mirrors}
\end{figure}
\hspace{1cm}
\newline
\newline

\subsubsection{Three Loop Invariant} We prove that the three loop invariant does not detect any of the aforementioned geometric symmetries.

\begin{lemma} \label{lemm_3_loop} Let $K$ be an oriented virtual knot. Then: 
\[
\phi_{ijk} (K) = \phi _{ijk} (K^{-1})=\phi_{ijk} (\overline{K})=\phi_{ijk} (K^s).
\]
\end{lemma}

\begin{proof}
Let $D_K$ be a Gauss diagram of the oriented knot $K$. A Gauss diagram for $K^s $ is obtained by changing the sign and direction of every arrow in $D_K $. We denote this Gauss diagram as $D_{K^s} $. If $y$ is an arrow of $D_K$, we denote its corresponding arrow of $D_{K^s}$ as $ \overline{y} $. The set of arrows that intersect $ \overline{y} $ are the same as those that intersect
$ y$.  The $ sign(y) = - sign(\overline{y} )$, but the the relative orientation of the intersection is the same in both diagrams. As a result, the index of arrows changes by sign and there is no change in the absolute value of the index. Hence, $ \phi _{ijk} (K) = \phi _{ijk} (K^s)$. 
\newline
\newline
\noindent Next consider $K^{-1}$. A Gauss diagram $D_{K^{-1}}$ of $K^{-1}$ is obtained from $D_K $ by reversing the direction of travel around the circle of the Gauss diagram.  The intersections between the arrows do not change, but the orientation of every intersection changes. If $y$ points from left to right relative to an upward pointing $x$ then $\overline{y} $ points from right to left. Hence there is no change to the absolute value of the index. Then $\phi_{ijk} (K) = \phi_{ijk} (K^{-1})$.
\newline
\newline
\noindent A Gauss diagram $D_{\overline{K}}$ for $ \overline{K}$ can be obtained by changing the sign of each arrows in $D_K $. As in the previous cases, it follows that $\phi_{ijk} (K) = \phi_{ijk} (\overline{K}) $. 
\end{proof}

\subsubsection{Three Loop Framed Invariant} On the other hand, the three loop framed invariant \emph{can} be used to detect geometric symmetries. 
\newline
\newline
Before proving this fact, we need a lemma and some additional notation. In $\mathbb{Z}_2^{\infty}$, label each of the copies of $\mathbb{Z}_2$. Hence, we have $\displaystyle{\mathbb{Z}_2^{\infty}=\bigoplus_{i=-\infty}^{\infty} (\mathbb{Z}_2)_i}$. Let $m:\mathbb{Z}_2^{\infty} \to \mathbb{Z}_2^{\infty}$ be the homomorphism which maps coordinate $i$ to coordinate $i-1$. It is clear that $m$ is an isomorphism.
 
\begin{lemma} Let $K$ be an oriented virtual knot invariant. Then:
\[
m \circ f \circ \Phi ^{fr} (K)=f \circ \Phi ^{fr}(K^{-1})=f \circ \Phi ^{fr}(\overline{K})=f \circ \Phi(K^s).
\]
\end{lemma}
\begin{proof} We prove the first equality only and leave the others as an exercise. Let $D_{K^{-1}}$ be a Gauss diagram of the oriented knot $K^{-1}$ as obtained in the proof of Lemma \ref{lemm_3_loop}.  We see again that, for all arrows $x$ of $D_K$, the weight of $\overline{x}$ in $D_{K}^{-1}$ is of opposite sign.
\newline
\newline
\noindent Also note that there is a one-to-one correspondence between pairs of non-intersecting arrows in $D_K$ and pairs of non-intersecting arrows in $D_K^{-1}$.
\newline
\newline
\noindent Finally, we note that $m \circ f (A_{i,i+1})=f(A_{-i-1,-i})$. Since $A$ is generated by the elements $A_{i,i+1}$, $-\infty<i<\infty$, this observation suffices to prove the lemma.
\newline
\newline
The other relations follow from similar investigations of the symmetries of Gauss diagrams. 
\end{proof}

\noindent The following theorem shows that the three loop invariant can be used to detect when a virtual knot is not equivalent to its inverse, not equivalent to its mirror image (see \cite{cheng_knot,cheng_link} for a similar result for the Cheng invariants), and not equivalent to its switch.

\begin{theorem} Let $K$ be an oriented virtual knot. If $f\circ\Phi ^{fr}(K) \ne m \circ f \circ \Phi ^{fr}(K)$, then $K$ is not equivalent as a virtual knot to its inverse, not equivalent to its mirror image, and not equivalent to its switch.
\end{theorem}
\begin{proof} Clearly, $\emph{writhe}(K)=\emph{writhe}(K^{-1})$. Therefore, if $f(\Phi^{fr}(K)) \ne m \circ f \circ \Phi^{fr}(K)$, we must have that $\Phi^{fr}(K) \ne \Phi^{fr}(K^{-1})$. By Theorem \ref{writheandframed}, we have that $K$ and $K^{-1}$ are inequivalent as virtual knots.
\newline
\newline
\noindent Since $\emph{writhe}(K)=-\emph{writhe}(\overline{K})=-\emph{writhe}(K^s)$, $\emph{writhe}(K)\equiv \emph{writhe}(\overline{K}) \equiv \emph{writhe}(K^s) \pmod 2$. Hence, by Corollary \ref{mod2frame}, we have that $K$ and $\overline{K}$ are inequivalent. Additionally, $K$ and $K^s$ are inequivalent.
\end{proof}

\noindent \textbf{Example:} For the oriented virtual knot below, we have that $f \circ \Phi  ^{fr}(K) \ne m \circ f \circ \Phi ^{fr} (K)$. The theorem then states that $K$ is not equivalent to its inverse, not equivalent to its mirror image, and equivalent to its switch.
\[
K=\begin{array}{c}\scalebox{.33}{\psfig{figure=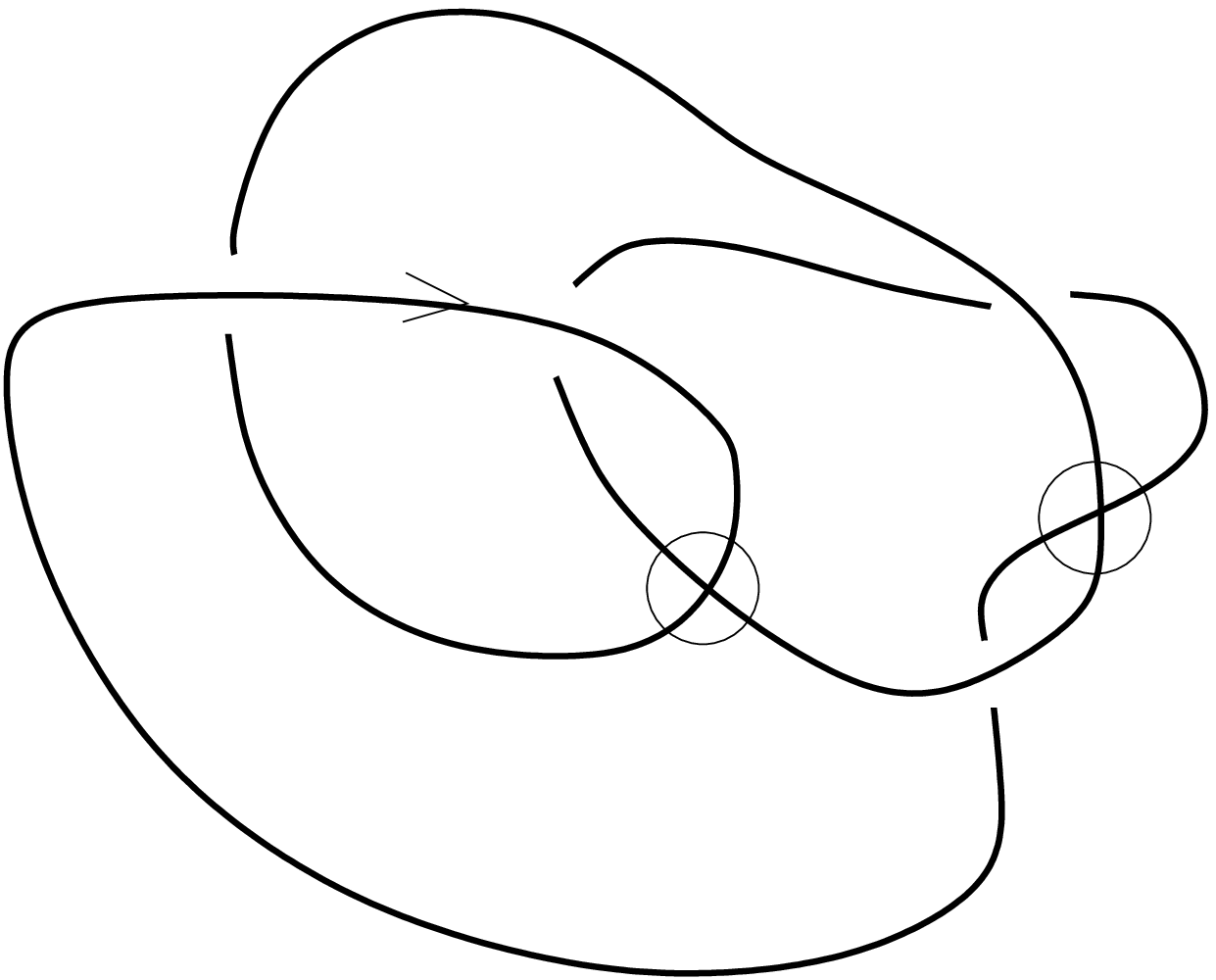}}\end{array}
\]

\subsection{Comparison with other invariants}
Let $D$ be a Gauss diagram and $x$ an arrow of $D$. By a \emph{way virtualization move} at $x$, we mean the diagram $D'$ which is obtained from $D$ by changing the way which $x$ points (i.e. it's direction). By a \emph{sign virtualization move} at $x$, we mean the diagram $D'$ which is obtained from $D$ by changing the sign of $x$ but not its direction. The Kauffman bracket is invariant under the way virtualization move \cite{kvirt}. The involutory quandle is invariant under the sign virtualization move \cite{kvirt}.

\subsubsection{Three Loop Invariant} We consider the Gauss diagram $D_1 $. The diagram $D_2$ is obtained from $D_1$ by a way virtualization. The diagram $D_3$ is obtained from $D_1$ by sign virtualization.  We calculate $ \phi_{203} $ for each of these diagrams and show that 
$\phi_{ijk} $ can detect crossing changes, and both sign and way virtualization. Hence, the three loop invariant can separate virtual knots which cannot be separated by the Kauffman bracket and can separate virtual knots which cannot be separated by the involutory quandle.

\begin{figure}[htb]
\[
\begin{array}{cc}
D_1=\begin{array}{c}\scalebox{.5}{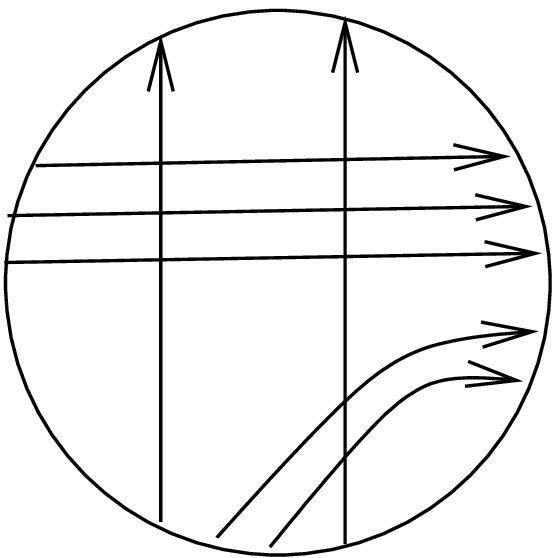} \end{array}, & D_2=\begin{array}{c}\scalebox{.5}{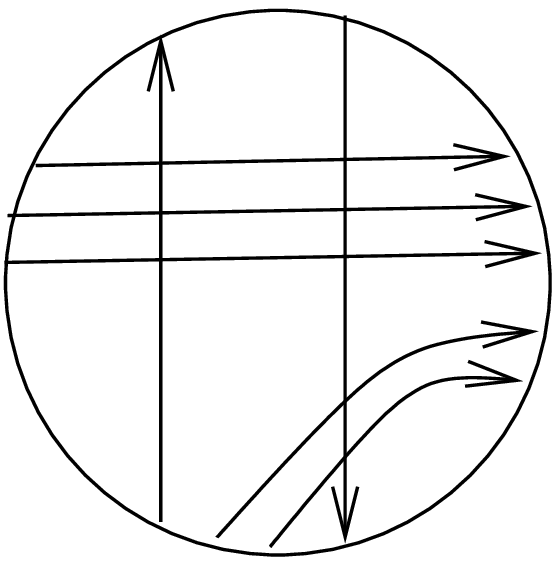}\end{array}, \\
D_3=\begin{array}{c}\scalebox{.5}{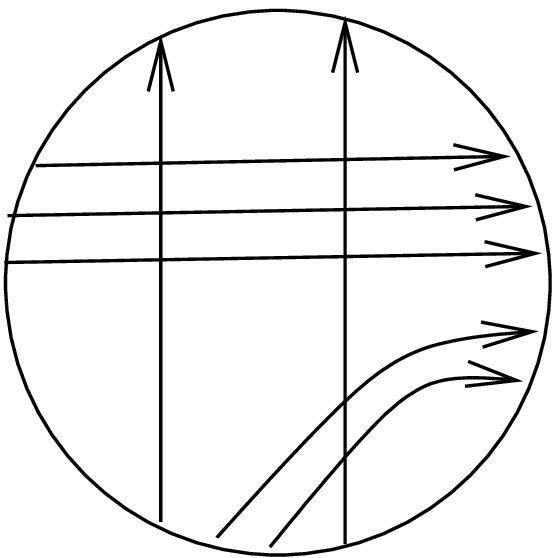} \end{array} & \\
\end{array}
\]
\caption{The three loop invariant detects both kinds of virtualizations: $\phi_{203}(D_1)=1,\,\,\phi_{203}(D_2)=0,\,\,\phi_{203} (D_3)=-1$.} \label{fig_three_loop_virt}
\end{figure}

\subsubsection{Three Loop Framed Invariant} The three loop framed invariant can detect both way virtualization and sign virtualization moves.  Hence, the three loop framed invariant can separate virtual knots which cannot be separated by the Kauffman bracket and can separate virtual knots which cannot be separated by the involutory quandle. In Figure \ref{fig_three_loop_fr_virt}, $D_2$ is obtained from $D_1$ by a way virtualization. $D_3$ is obtained from $D_1$ by  sign virtualization.

\begin{figure}[htb]
\[
\begin{array}{cc}
D_1=\begin{array}{c}\scalebox{.35}{\psfig{figure=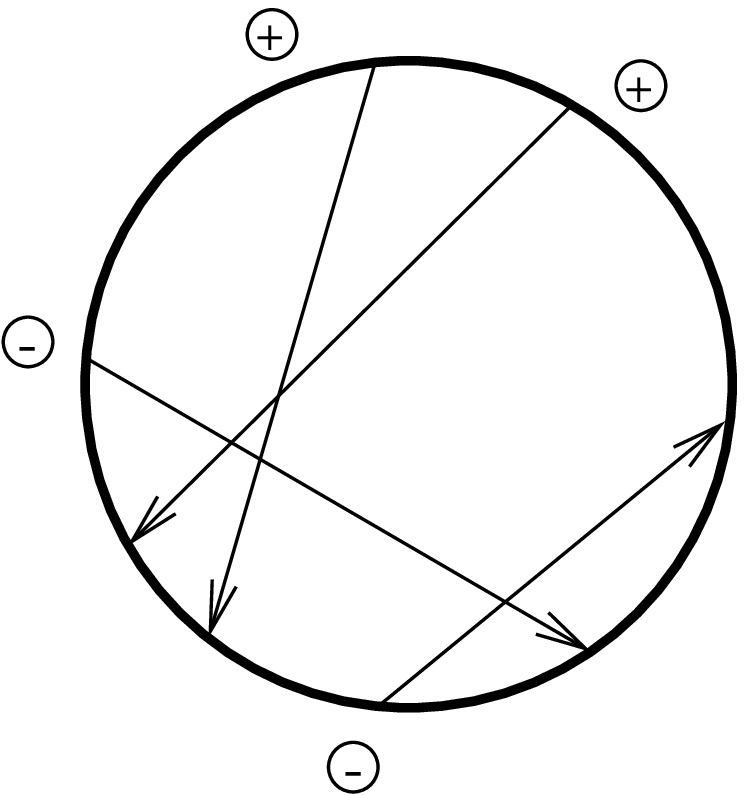}} \end{array}, & D_2=\begin{array}{c}\scalebox{.35}{\psfig{figure=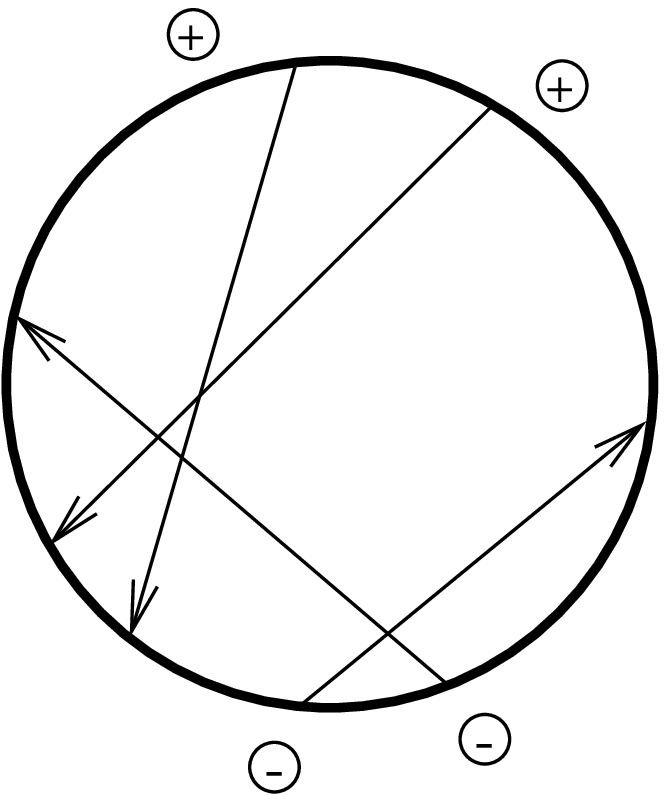}}\end{array}, \\
D_3=\begin{array}{c}\scalebox{.35}{\psfig{figure=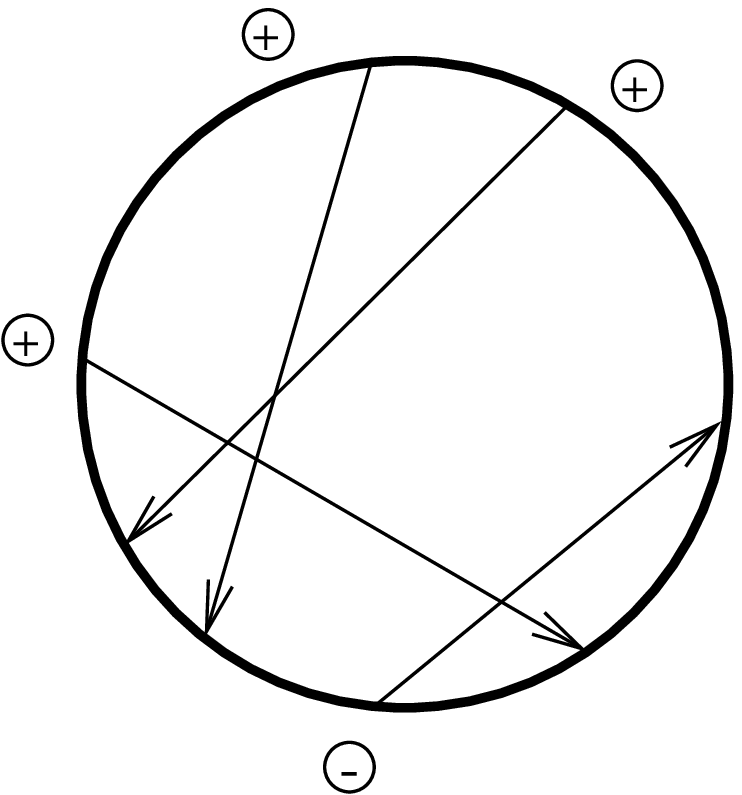}} \end{array} & \\
\end{array}
\]
\caption{The three loop framed invariant detects both kinds of virtualizations: $\Phi ^{fr}(D_1)=A_{-2,0},\,\,\Phi ^{fr}(D_2)=A_{0,2},\,\,\Phi ^{fr}(D_3)=A_{0,2}$.} \label{fig_three_loop_fr_virt}
\end{figure}

\section{Formalization of ``Analogues''} \label{sec_gvinariant} 

\noindent The previous two sections have discussed the three loop isotopy invariant and the three loop framed isotopy invariant. They have been referred to as the virtual knot ``analogues'' of the Grishanov-Vassiliev invariants of order 2 for knots in thickened surfaces. The term ``analogue'' has been used informally throughout. In the present section, we present a formal definition of ``analogue'' and prove using this definition that the $\phi_{ijk}$ are indeed virtual knot analogues of the Grishanov-Vassiliev invariants of order 2.

\subsection{Definition of Analogue} We begin with some terminology. Let $\Sigma$ be a closed oriented surface. Let $\mathscr{S}$ denote the set of such $\Sigma$.
\newline
\newline
Let $\Sigma \in \mathscr{S}$. Let $\tau$ be a knot in $\Sigma \times I$.  As follows in the usual knot theory of $\mathbb{R}^3=\mathbb{R}^2 \times \mathbb{R}$, knots in $\Sigma \times I$ may be represented as knot diagrams on $\Sigma$. Crossings are represented in the traditional sense and oriented knot diagrams on $\Sigma$ will have local writhe numbers. If $\tau$ is an oriented knot diagram on $\Sigma$, one may find a Gauss diagram $D_{\tau}$ of $\tau$ using the exact same procedure as that for virtual knots.
\newline
\newline
By the pair $(\Sigma,\tau)$, we denote either a knot in $\Sigma \times I$ or a knot diagram on $\Sigma$. 
\newline
\newline
\noindent Knots in $\Sigma \times I$ are considered equivalent up to (smooth or p.l) ambient isotopy of $\Sigma \times I$. On the other hand, knot diagrams on $\Sigma$ are considered equivalent up the Reidemeister moves, where each small ball defining a move depicts a neighborhood of the arcs which is homeomorphic to a small ball on $\Sigma$. Let $\mathscr{K}(\Sigma)$ denote either set of equivalence classes of knots in $\Sigma \times I$ or equivalence classes of knot diagrams on $\Sigma$. By the obvious generalization of the Reidemeister theorem to thickened surfaces, we see that $\mathscr{K}(\Sigma)$ is well-defined.

\begin{definition} Let $\Sigma\in \mathscr{S}$ and $(\Sigma,\tau)$ an oriented knot diagram on $\Sigma$. Let $D_{\tau}$ be a Gauss diagram of $(\Sigma,\tau)$. Let $K_{\tau}$ be a \emph{virtual knot} having $D_{\tau}$ as a Gauss diagram. Let $\varphi_{\Sigma}:\mathbb{Z}[\mathscr{K}(\Sigma)] \to \mathscr{K}$ be defined on generators by $\varphi_{\Sigma}(\Sigma,\tau)=K_{\tau}$. Note that $\varphi_{\Sigma}$ is well-defined \cite{kamkam}: If $\tau_1$ and $\tau_2$ are Reidemeister equivalent on $\Sigma$, then $\varphi_{\Sigma}(\Sigma,\tau_1)\leftrightarrow \varphi_{\Sigma}(\Sigma,\tau_2)$. 
\end{definition}

\begin{definition}[Virtual Knot Analogue] Let $\{A_{\Sigma}|\Sigma \in \mathscr{S} \}$ be a family of abelian groups. Let $\mathscr{F}=\{f[\Sigma]:\mathbb{Z}[\mathscr{K}(\Sigma)] \to A_{\Sigma}| \Sigma \in \mathscr{S} \}$ be a family of given knot invariants, where one invariant is specified for each $\Sigma \in \mathscr{S}$. Let $G$ be an abelian group. We say that $\nu:\mathbb{Z}[\mathscr{K}] \to G$ is a \emph{virtual knot analogue} of $\mathscr{F}$ if there is a set of group homomorphisms $\{\hat{\varphi}_{\Sigma}:A_{\Sigma} \to G|\Sigma \in \mathscr{S}\}$ such that for all $\Sigma \in \mathscr{S}$, the following diagram commutes: 
\[
\xymatrix{
\mathbb{Z}[\mathscr{K}(\Sigma)] \ar[r]^{f[\Sigma]} \ar[d]_{\varphi_{\Sigma}} & A_{\Sigma} \ar[d]^{\hat{\varphi}_{\Sigma}}\\
\mathbb{Z}[\mathscr{K}] \ar[r]_{\nu} & G \\
}
\]    
\end{definition}

\noindent We are interested in finding a virtual knot analogue for the Grishanov-Vassiliev invariants of order 2. We will do this by first introducing a single invariant of knots in a thickened surface $\Sigma$, $\Phi[\Sigma]:\mathbb{Z}[\mathscr{K}(\Sigma)] \to \mathscr{A}(\Sigma)$, which generalizes all of the Grishanov-Vassiliev invariants of order 2. The invariant $\Phi[\Sigma]$ is a generalization in the sense that if $\Phi$ is a Grishanov-Vassiliev invariant of order 2 on $\Sigma$, then $\Phi$ can be represented as $\Phi=\hat{\Phi} \circ \Phi[\Sigma]$, where $\hat{\Phi}$ is a linear functional $\hat{\Phi}:\mathscr{A}(\Sigma) \to \mathbb{Z}$. Next we will introduce an invariant of virtual knots $\phi:\mathbb{Z}[\mathscr{K}] \to \mathscr{A}$ which generalizes all of the virtual knot invariants $\phi_{i,j,k}$. Again, $\phi$ is a generalization of $\phi_{i,j,k}$ in the sense that there is a linear functional $\hat{\phi}_{i,j,k}:\mathscr{A} \to \mathbb{Z}$ such that $\phi_{i,j,k}=\hat{\phi}_{i,j,k}\circ\phi$.
\newline
\newline
After these invariants are constructed, it will be shown that $\phi:\mathbb{Z}[\mathscr{K}] \to \mathscr{A}$ is the virtual knot analogue of the family $\{\Phi[\Sigma]:\mathbb{Z}[\mathscr{K}(\Sigma)] \to \mathscr{A}(\Sigma)| \Sigma \in \mathscr{S}\}$.
 
\subsection{Definition of $\Phi[\Sigma]$} \label{sec_gvinvariants} Gauss diagram invariants and finite-type invariants of knots in thickened surfaces have been extensively studied (see for example \cite{Fd,arnaud,GrVa,vassily_curves}). In the present section, we present a generalization of the Grishanov-Vassiliev invariants of order 2. First we define a function which assigns to each knot diagram $K$ on $\Sigma$ a formal sum of Gauss diagrams. We then define a group $\mathscr{A}(\Sigma)$ which is generated by a set of Gauss diagrams. It is then proved that if $\Phi[\Sigma]$ is evaluated in $\mathscr{A}(\Sigma)$, then $\Phi[\Sigma]$ is an invariant of knots on $\Sigma \times I$. Lastly, it is shown that any Grishanov-Vassiliev invariant of order 2 can be represented as an integer valued linear functional on $\mathscr{A}(\Sigma)$.
\newline
\newline
Let $\Sigma \in \mathscr{S}$ and $K$ be a knot diagram on $\Sigma$. Let $D$ be a Gauss diagram of $K$. Let $D_{||}'$ be a subdiagram of $D$ consisting of two non-intersecting arrows. Consider the circle $S^1$ of $D$ to be the boundary of the disk $D^2$. Then the arrows of $D_{||}'$ divide the disk $D^2$ into three connected components called the \emph{regions} of $D_{||}'$.
\newline
\newline
We label the regions of $D_{||}'$ by giving the oriented smoothing at the corresponding crossings of $K$. This divides $K$ into three oriented closed curves $C_1,C_2,C_3$, each of which corresponds to a region $D_1,D_2,D_3$ of $D_{||}'$. Now, each $C_i$ represents some element $\alpha_i$ of $H_1(\Sigma;\mathbb{Z})$. The region $D_i$ of $D$ is labeled with the homology class $\alpha_i$. We define:
\[
\Phi[\Sigma](D)=\sum_{D_{||}'\subset D} sign(D_{||}') D_{||}',
\]
where the regions of each summand are labeled as described above, and $sign(D_{||}')$ is the product of the signs of the arrows in $D_{||}'$. Also in the formal sum, we will delete all sign information in the Gauss diagram (i.e. no arrows are labeled $\oplus$ or $\ominus$).
\newline
\newline
\noindent For $\Sigma \in \mathscr{S}$ we define an abelian group $\mathscr{A}(\Sigma)$. The set of generators of $\mathscr{A}(\Sigma)$ is the set $T(\Sigma)$ of unsigned Gauss diagrams given below. The set of relations of $\mathscr{A}(\Sigma)$ is the set $R(\Sigma)$ of unsigned Gauss diagrams given below. 

\begin{eqnarray*}
T(\Sigma) &=& \left\{\left.
\begin{array}{c}
\scalebox{.3}{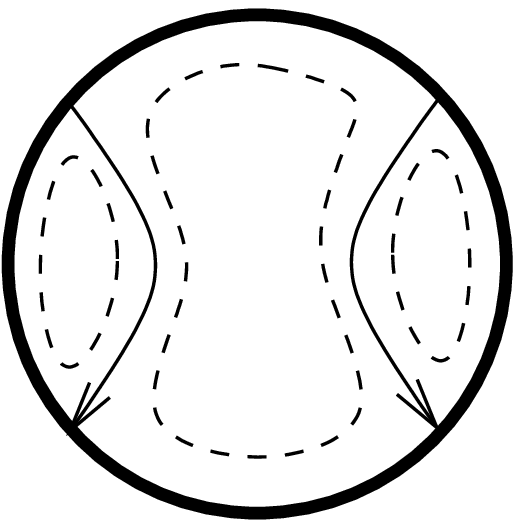}
\end{array}, \begin{array}{c}
\scalebox{.3}{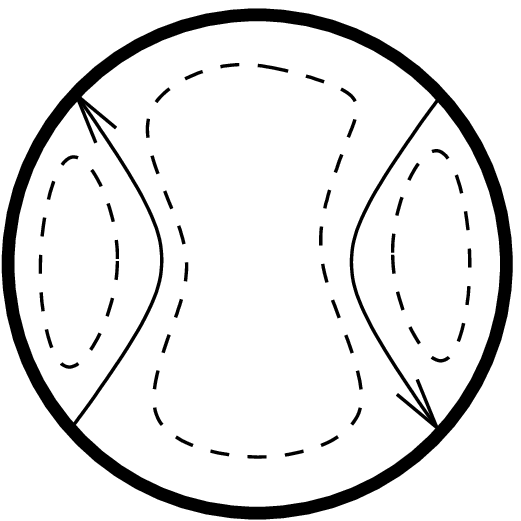}
\end{array},\begin{array}{c}
\scalebox{.3}{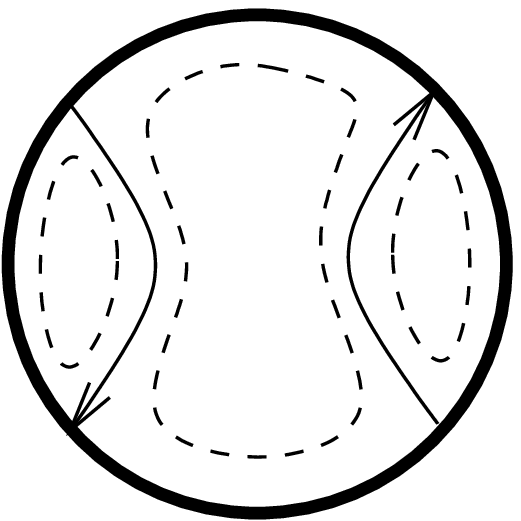}
\end{array} \right| \alpha,\beta,\gamma \in H_1(\Sigma ;\mathbb{Z}) \right\} \\
R(\Sigma) &=& \left\{\left.
\begin{array}{c}
\scalebox{.3}{\input{gv_invariant_1.pstex_t}}
\end{array}, \begin{array}{c}
\scalebox{.3}{\input{gv_invariant_2.pstex_t}}
\end{array},\begin{array}{c}
\scalebox{.3}{\input{gv_invariant_3.pstex_t}}
\end{array}\right| 0 \in \{\alpha,\beta,\gamma\} \right\} \\
          & \cup & \left\{\begin{array}{c}
\scalebox{.3}{\input{gv_invariant_1.pstex_t}}
\end{array}+\begin{array}{c}
\scalebox{.3}{\input{r1_2.pstex_t}}
\end{array}-\begin{array}{c}
\scalebox{.3}{\input{r1_3.pstex_t}}
\end{array} \right\} \\
          & \cup & \left\{ \begin{array}{c}
\scalebox{.3}{\input{gv_invariant_1.pstex_t}}
\end{array}+\begin{array}{c}
\scalebox{.3}{\input{r2_2.pstex_t}}
\end{array}-\begin{array}{c}
\scalebox{.3}{\input{r2_3.pstex_t}}
\end{array}\right\}
\end{eqnarray*}

\begin{lemma} \label{lemma_i_sigma} The map $\Phi[\Sigma]:\mathbb{Z}[\mathscr{K}(\Sigma)] \to \mathscr{A}(\Sigma)$ is an invariant of knots in $\Sigma \times I$.
\end{lemma}
\begin{proof} We will show that $\Phi[\Sigma]$ is invariant under all Reidemeister 1 and 2 moves and one of the two types of Reidemeister 3 moves.  By \"{O}stlund's theorem \cite{Ost}, this is sufficient to prove the lemma.
\newline
\newline
\underline{Reidemeister 1 Moves:} Let $D_{||}'$ be a subdiagram on the left hand side of an Reidemeister $1$ move having two arrows. If neither of the arrows corresponds to the crossing in the move, then $D_{||}'$ also appears a subdiagram the right hand side of the move with identical labels. If $D_{||}'$ contains the arrow which is involved in the move, then $D_{||}'$ must have a homology class which is $0$. Hence $D_{||}' \in R(\Sigma)$. It follows that $\Phi[\Sigma]$ is invariant under the Reidemeister $1$ move.
\newline
\newline
\underline{Reidemeister 2 Moves:} The left hand side of a Reidemeister $2$ move has two arrows, one of which is signed $\oplus$ and one of which is signed $\ominus$. Let $D_{||}'$ be a subdiagram containing two non-intersecting arrows. If neither arrow of $D_{||}'$ is involved in the move, then $D_{||}'$ appears as a subdiagram on the right hand side of the move. If one of the arrows is involved in the move, then $D_{||}'$ also appears as subdiagram with opposite sign. Hence $D_{||}'-D_{||}'=0$ and the total  contribution is $0$ to the left hand side of the Reidemeister 2 move. There are two cases if both arrows of $D_{||}$ are involved in the move: the two arrows do not intersect and two arrows intersect.  If the two arrows intersect, there is no contribution on the left hand side of the move. If the two arrows do not intersect, then one of the homology classes must be $0$. Hence $D_{||}' \in R(\Sigma)$ in this case. It follows that $\Phi[\Sigma]$ is invariant under the Reidemeister $2$ move.
\newline
\newline
\underline{Reidemeister 3 Moves:} Since all Reidemeister 1 and 2 moves have been considered, it is sufficient to consider only the move in Figure \ref{omega3} and the move obtained from this move by switching each of the depicted crossings. Let $\Omega 3$ denote the move depicted in Figure \ref{omega3}, $\text{LHS}(\Omega 3)$ the figure depicted on the left hand side of Figure \ref{omega3}, and $\text{RHS}(\Omega 3)$ the figure depicted on the right hand side of Figure \ref{omega3}.
\newline
\newline
\noindent Let $D_{||}'$ be a subdiagram on the left hand side of the move.  It is easy to see that if neither or only one of the arrows of $D_{||}'$ are involved in the move, then $D_{||}'$ is also an subdiagram of the right hand side of the move.  It remains only to consider the case in which both arrows of $D_{||}'$ are involved in the move.  
\newline
\newline
In Figure \ref{fig_omega_3_hom}, the three figures depict the three ways in which two crossings can be smoothed on the left and right hand sides of the Reidemeister $3$ move so that the corresponding crossings are non-intersecting. The labels $\delta_1$, $\delta_2$, $\delta_3$ represent homology classes in $H_1(\Sigma;\mathbb{Z})$. With these observations, we observe that $\Phi[\Sigma](\text{LHS}(\Omega 3))=\Phi[\Sigma](\text{RHS}(\Omega 3))$ reduces to:
\[
\begin{array}{c}
\scalebox{.4}{\input{i_is_good_plus_1.pstex_t}}
\end{array}+\begin{array}{c}
\scalebox{.4}{\input{i_is_good_plus_2.pstex_t}}
\end{array}=\begin{array}{c}
\scalebox{.4}{\input{i_is_good_plus_3.pstex_t}}
\end{array}
\]
When the depicted crossings in the Reidemeister $3$ move are reversed, $\Phi[\Sigma](\text{LHS}(\Omega 3))=\Phi[\Sigma](\text{RHS}(\Omega 3))$ reduces to the following equation:
\[
\begin{array}{c}
\scalebox{.4}
{\input{i_is_good_plus_1.pstex_t}}
\end{array}+\begin{array}{c}
\scalebox{.4}
{\input{i_is_good_minus_2.pstex_t}}
\end{array}=\begin{array}{c}
\scalebox{.4}
{\input{i_is_good_minus_3.pstex_t}}
\end{array}
\]
As both of these relations are in $R(\Sigma)$, it follows that $\Phi[\Sigma]$ is an invariant under the third Reidemeister move.  This completes the proof of the lemma.
\end{proof}

\begin{figure}[htb]
\[
\begin{array}{ccc}
\begin{array}{c}
\scalebox{.3}{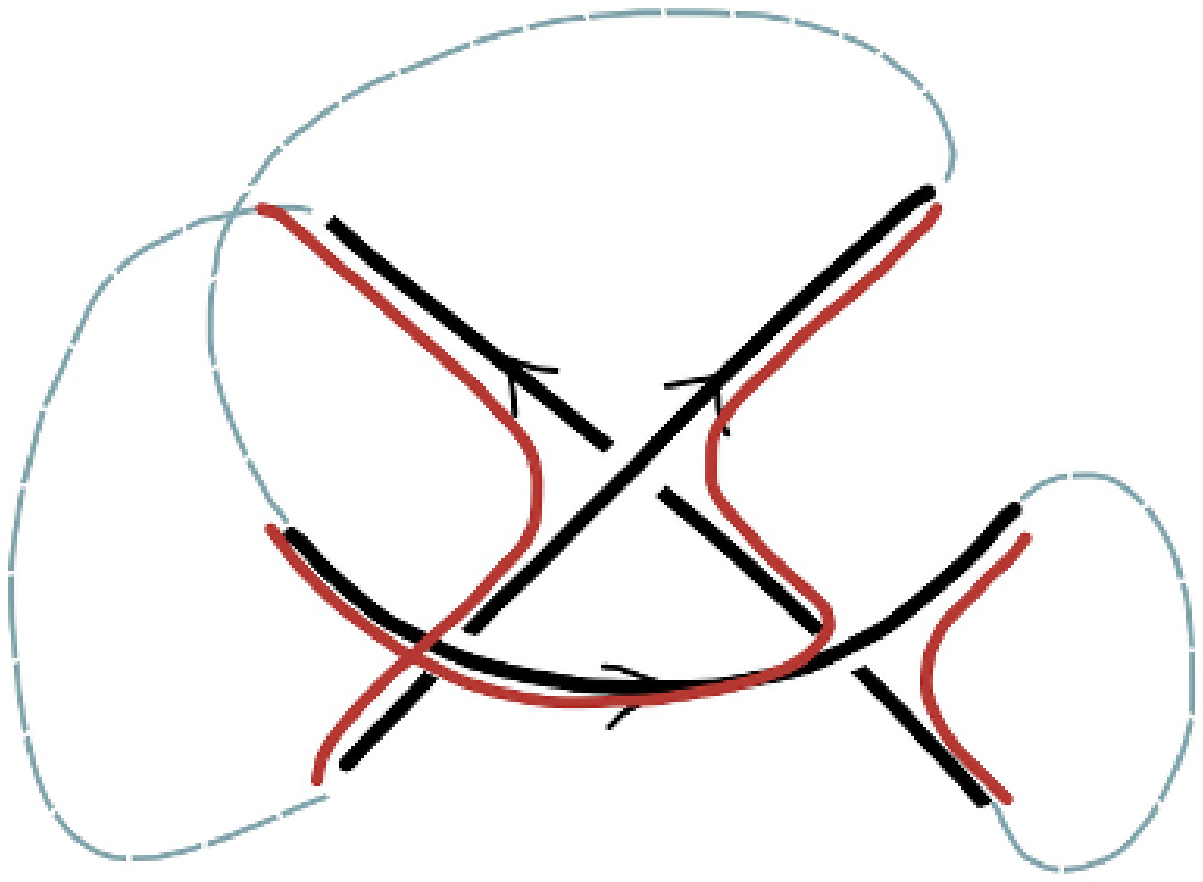}
\end{array}
&
\begin{array}{c}
\scalebox{.3}{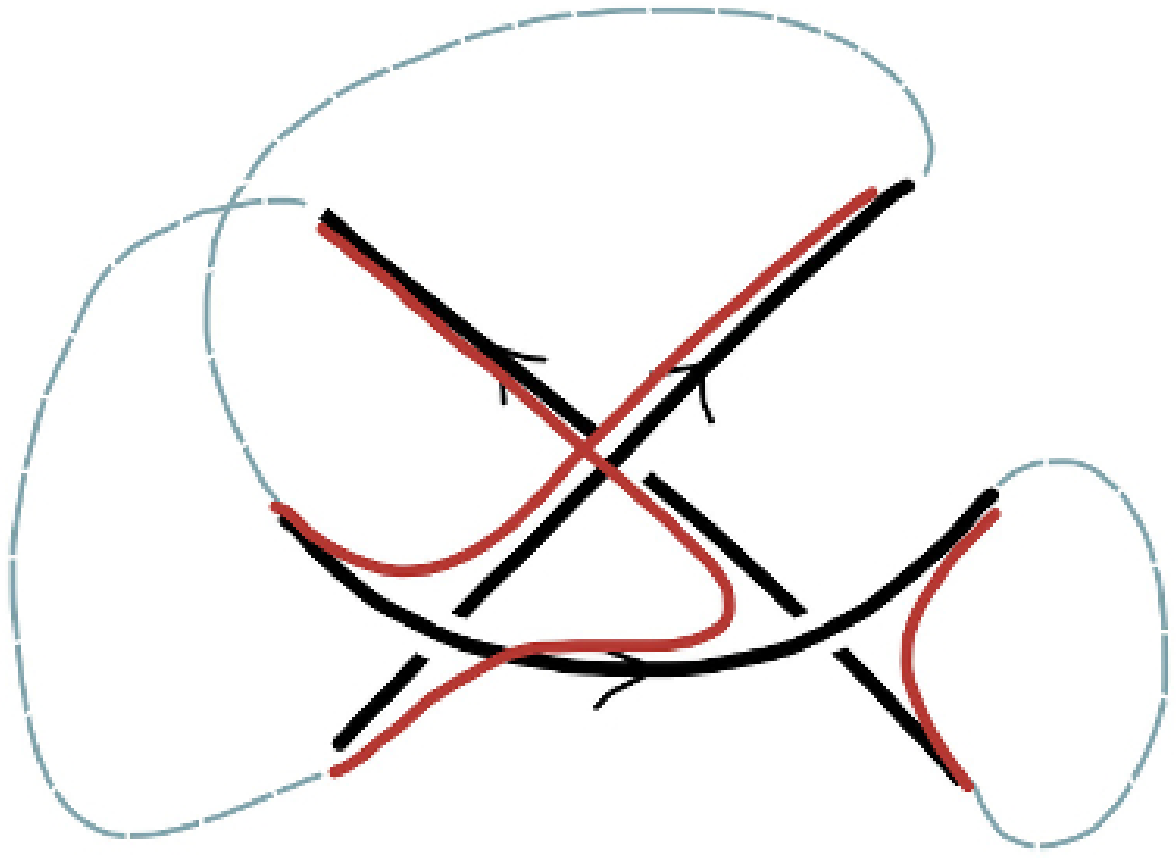}
\end{array}
&
\begin{array}{c}
\scalebox{.3}{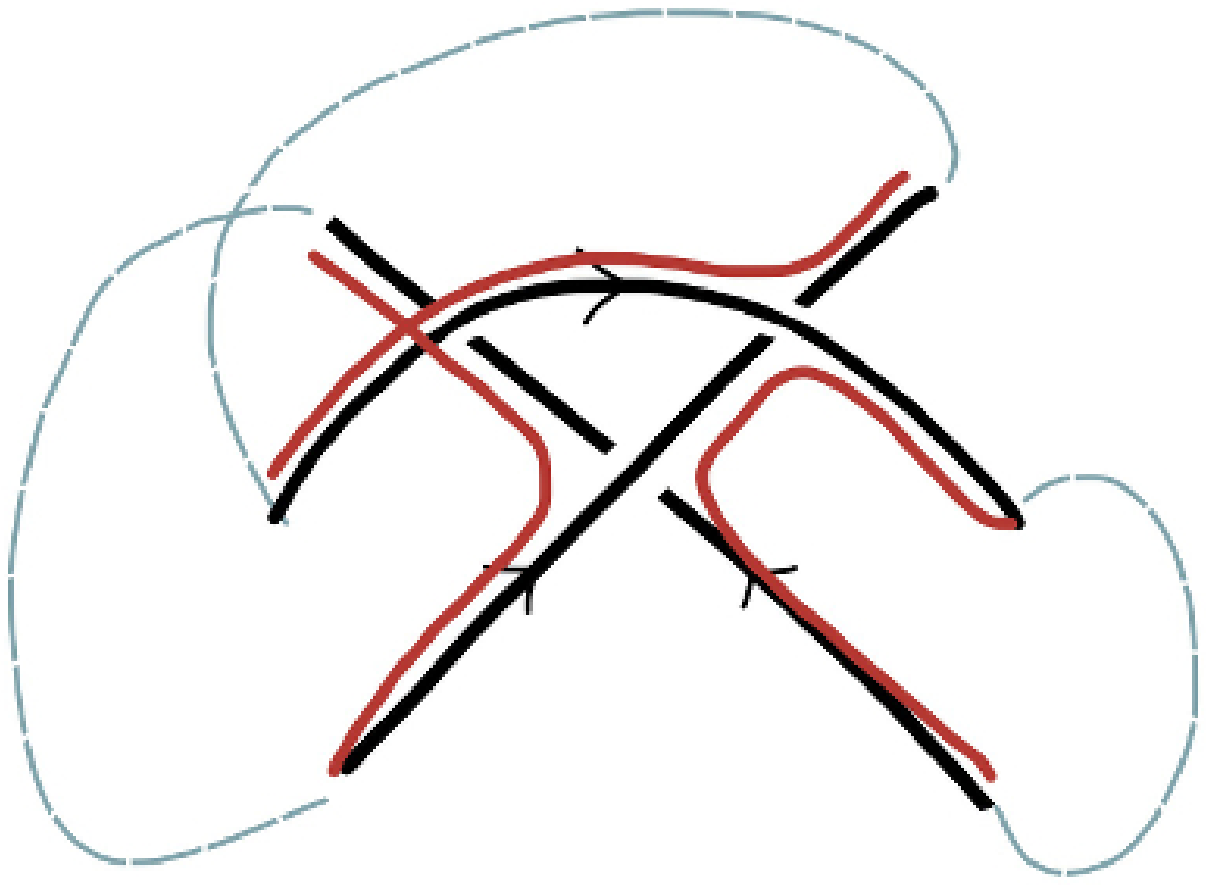}
\end{array}
\end{array}
\]
\caption{Possible ways to smooth along two parallel arrows in the diagrams of an Reidemeister $3$ move.} \label{fig_omega_3_hom}
\end{figure}

\noindent We now recall the definition of the Grishanov-Vassiliev invariants of order 2. For simplicity, we will consider only the case in which the regions are labeled with elements of $H_1(\Sigma;\mathbb{Z})$. Let $\alpha,\beta,\gamma \in H_1 (\Sigma;\mathbb{Z})$ such that $0 \notin \{\alpha,\beta,\gamma\}$ and either (1) $\alpha$, $\beta$, and $\gamma$ are all distinct or (2) $\alpha=\gamma$ and $\beta \ne \gamma$. Then we may define the Grishanov-Vassiliev invariant of order 2, $\Phi_{\alpha,\beta,\gamma}:\mathbb{Z}[\mathscr{K}(\Sigma)] \to \mathbb{Z}$, as follows:
\[
\Phi_{\alpha,\beta,\gamma}(K)=\left<\begin{array}{c}
\scalebox{.4}{\input{gv_invariant_1.pstex_t}}
\end{array}+\begin{array}{c}
\scalebox{.4}{\input{gv_invariant_2.pstex_t}}
\end{array}+\begin{array}{c}
\scalebox{.4}{\input{gv_invariant_3.pstex_t}}
\end{array},\Phi[\Sigma](K)\right>
\] 
It was proved in \cite{GrVa} that under the conditions (1) and (2) on $\alpha,\beta,\gamma$, that $\Phi_{\alpha,\beta,\gamma}$ is an invariant of knots in $\Sigma \times I$. We remark that the definition of $\Phi_{\alpha,\beta,\gamma}$ in \cite{GrVa} is quite different. Our notation and the prior construction of $\Phi[\Sigma]$ have made the definition more compact. 
\newline
\newline
Define $F_{\alpha,\beta,\gamma}$ to be the formal sum of labeled diagrams in the definition of $\Phi_{\alpha,\beta,\gamma}$. Hence, we have that $\Phi_{\alpha,\beta,\gamma}(\cdot)=\left<F_{\alpha,\beta,\gamma},\Phi[\Sigma](\cdot)\right>$. Define $\hat{\Phi}_{\alpha,\beta,\gamma}(\cdot)=\left<F_{\alpha,\beta,\gamma},\cdot\right>$.
\newline
\newline
The following lemma shows that $\Phi[\Sigma]$ is a generalization of all the Grishanov-Vassiliev invariants of order 2.

\begin{lemma} \label{lemma_phi_sigma_func} The Grishanov-Vassiliev invariant of order 2 can be represented as an integer valued linear functional on $\mathscr{A}(\Sigma)$. In particular, we have that $\hat{\Phi}_{\alpha,\beta,\gamma} \in \text{Hom}_{\mathbb{Z}}(\mathscr{A}(\Sigma),\mathbb{Z})$ and $\Phi_{\alpha,\beta,\gamma}=\hat{\Phi}_{\alpha,\beta,\gamma} \circ \Phi[\Sigma]$. 
\end{lemma}
\begin{proof} The second claim follows immediately from our definition of $\Phi_{\alpha,\beta,\gamma}$. To establish the first claim, it will be shown that $\hat{\Phi}_{\alpha,\beta,\gamma}(R)=0$ for all $R \in R(\Sigma)$. If $D$ is one of the single diagrams in $R(\Sigma)$ having a region labeled with $0$, then the definition of $\hat{\Phi}_{\alpha,\beta,\gamma}$ implies that $\hat{\Phi}_{\alpha,\beta,\gamma}(D)=0$.
\newline
\newline
Suppose $R=D_1+D_2-D_3 \in R(\Sigma)$ and suppose that the homology classes labeling $D_1$ are, in order from left to right, $\delta_1,\delta_2,\delta_3$. It is clear that if $\{\alpha,\beta,\gamma\} \ne \{\delta_1,\delta_2,\delta_3\}$, then $\hat{\Phi}_{\alpha,\beta,\gamma}(R)=0$. Now suppose that $\{\alpha,\beta,\gamma\}=\{\delta_1,\delta_2,\delta_3\}$. Since, $\alpha \ne \beta$, it follows that either zero or one of $D_1$ and $D_2$ is counted by $\hat{\Phi}_{\alpha,\beta,\gamma}$. If none are counted, then $D_3$ is not counted.  If one is counted, then $D_3$ is also counted and the contribution to the value of $\hat{\Phi}_{\alpha,\beta,\gamma}$ is $1-1=0$. Hence, $\hat{\Phi}_{\alpha,\beta,\gamma}(R)=0$. 
\end{proof}

\subsection{Definition of $\phi$} The previous subsection introduced an invariant $\Phi[\Sigma]$ of knots in a given thickened surface $\Sigma \times I$ which is a generalization of all Grishanov-Vassiliev invariants of order 2. In this section, we introduce an abelian group $\mathscr{A}$ and an invariant $\phi:\mathbb{Z}[\mathscr{K}] \to \mathscr{A}$ which is a generalization of all the invariants $\phi_{i,j,k}$.  We define $\mathscr{A}=\frac{\mathbb{Z}[T]}{\left<R\right>}$, where $T$ and $R$ are the sets given below:

\begin{eqnarray*}
T &=& \left\{ \left. \begin{array}{c}
\scalebox{.4}{\input{phi_ijk_1.pstex_t}}
\end{array},\begin{array}{c}
\scalebox{.4}{\input{phi_ijk_2.pstex_t}}
\end{array},\begin{array}{c}
\scalebox{.4}{\input{phi_ijk_3.pstex_t}}
\end{array} \right| i,j,k \in \mathbb{N} \cup \{0\} \right\} \\
\end{eqnarray*}
\begin{eqnarray*}
R &=& \left\{\left.\begin{array}{c}
\scalebox{.4}{\input{phi_ijk_1.pstex_t}}
\end{array},\begin{array}{c}
\scalebox{.4}{\input{phi_ijk_2.pstex_t}}
\end{array},\begin{array}{c}
\scalebox{.4}{\input{phi_ijk_3.pstex_t}}
\end{array} \right| |\{i,j,k\}| <3 \right\} \\
& \cup & \left\{ \begin{array}{c}
\scalebox{.4}{\input{phi_omega_rels_1_1.pstex_t}}
\end{array}+\begin{array}{c}
\scalebox{.4}{\input{phi_omega_rels_1_2.pstex_t}}
\end{array} - \begin{array}{c}
\scalebox{.4}{\input{phi_omega_rels_1_3.pstex_t}} \end{array} \right\} \\
  & \cup & \left\{ \begin{array}{c}
\scalebox{.4}{\input{phi_omega_rels_1_1.pstex_t}}
\end{array}+\begin{array}{c}
\scalebox{.4}{\input{phi_omega_rels_2_2.pstex_t}}
\end{array}-\begin{array}{c}
\scalebox{.4}{\input{phi_omega_rels_2_3.pstex_t}} \end{array}\right\} 
\end{eqnarray*}

\noindent We now define the virtual knot invariant $\phi:\mathbb{Z}[\mathscr{K}] \to \mathscr{A}$. Let $D$ be a Gauss diagram of a virtual knot diagram $K$. We denote by $\phi(K)$ the formal sum of subdiagrams of $D$ having two non-intersecting arrows whose regions are given by relative weights:
\[
\phi(K)=\sum_{D_{||}'\subset D} sign(D_{||}') D_{||}',
\]
where $sign(D_{||}')$ is the product of the signs of the arrows of $D_{||}'$. 
\newline
\newline
Let $F_{i,j,k}$ be the formal sum of Gauss diagrams depicted in Figure \ref{fig_phi_ijk_defn}. Define $\hat{\phi}_{i,j,k}:\mathscr{A} \to \mathbb{Z}$ by $\hat{\phi}_{i,j,k}(\cdot)=\left<F_{i,j,k},\cdot\right>$

\begin{lemma} The map $\phi:\mathbb{Z}[\mathscr{K}] \to \mathscr{A}$ is an invariant of virtual knots.
\end{lemma}
\begin{proof} This follows from the definitions and the proof that $\phi_{i,j,k}$ is an invariant of virtual knots.  As it is similar to arguments already presented, we omit the details.
\end{proof}

\noindent The following lemma shows that $\phi$ is indeed a generalization of all the invariants $\phi_{i,j,k}$ of virtual knots.

\begin{lemma} Let $i,j,k \in \mathbb{N} \cup \{0\}$ such that $i \ne j \ne k \ne i$. Then $\phi_{i,j,k}$ can be represented as an integer valued linear functional on $\mathscr{A}$. In particular, we have that $\hat{\phi}_{i,j,k} \in \text{Hom}_{\mathbb{Z}}(\mathscr{A},\mathbb{Z})$ and $\phi_{i,j,k}=\hat{\phi}_{i,j,k} \circ \phi$.
\end{lemma}
\begin{proof} This follows immediately from the definitions and the previous lemma.
\end{proof}
 
\subsection{$\phi$ is a virtual knot analogue of the set of $\Phi[\Sigma]$} The previous two sections introduced the invariant $\Phi[\Sigma]$ of knots in a given thickened surface and the invariant $\phi$ of virtual knots.  It was proved that $\Phi[\Sigma]$ is a generalization of all the Grishanov-Vassiliev invariants of order 2 and that $\phi$ is a generalization of all the invariants $\phi_{i,j,k}$. In this section, we use intersection theory to prove that $\phi$ is a virtual knot analogue of the family $\{\Phi[\Sigma]
|\Sigma \in \mathscr{S}\}$. More precisely, we show that for all $\Sigma \in \mathscr{S}$, that there is a $\hat{\varphi}_{\Sigma}:\mathscr{A}(\Sigma) \to \mathscr{A}$ such that $\hat{\varphi}_{\Sigma} \circ \Phi[\Sigma]=\phi \circ \varphi_{\Sigma}$. We begin by recalling some fundamentals of intersection theory.
\newline
\newline
\noindent Let $M$ and $N$ be sub-manifolds of a closed $n$-manifold $V$ such that $\dim(V)=\dim(M)+\dim(N)$. Suppose that $M \cap N$ is a finite set of points and that all intersections are transverse. Recall that the intersection number of $M$ and $N$ can be computed from homology classes $a$ and $b$ represented by $M$ and $N$, respectively. If $\mu' \in H^2(\Sigma \times \Sigma;\mathbb{Z})$ (see \cite{gh} Corollary (30.3)) represents the restriction of the Thom class, then the intersection number is given by $[a \times b,\mu']$ (see \cite{gh} Proposition (31.7)).
\newline
\newline
\noindent We now return to the definition of the map $\hat{\varphi}_{\Sigma}:\mathscr{A}(\Sigma) \to \mathscr{A}$. Let $\Sigma$ be oriented and let $\mu'\in H^2(\Sigma \times \Sigma;\mathbb{Z})$ denote the restriction of the Thom class of $\Sigma$ (see Corollary (30.3) \cite{gh}). Suppose $D_{||} \in T(\Sigma)$. If $\alpha$ and $\beta$ are homology classes of $D_{||}$, let $\omega_{\alpha,\beta}=|[\alpha \times \beta,\mu']|$ \cite{gh}. Then $\hat{\varphi}_{\Sigma}$ is defined on the generators of $\mathscr{A}(\Sigma)$ as shown in Figure \ref{fig_defn_hat_varphi}.

\begin{figure}
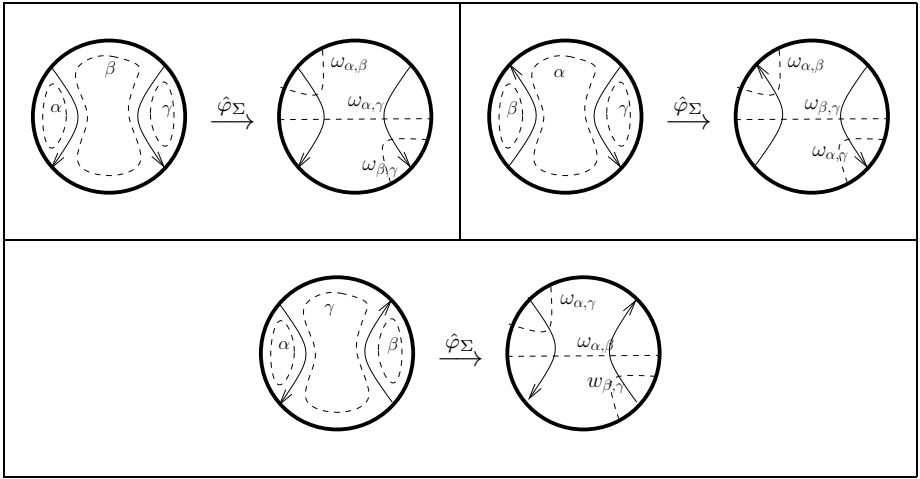

\[
\begin{array}{|c|c|} \hline
 &  \\
\begin{array}{c}
\scalebox{.4}{\input{gv_invariant_1.pstex_t}}
\end{array} \stackrel{\hat{\varphi}_{\Sigma}}{\longrightarrow} \begin{array}{c}
\scalebox{.4}{\input{defn_hat_phi_sigma_gens_1.pstex_t}}
\end{array} & 
\begin{array}{c}
\scalebox{.4}{\input{gv_invariant_2.pstex_t}}
\end{array} \stackrel{\hat{\varphi}_{\Sigma}}{\longrightarrow} \begin{array}{c}
\scalebox{.4}{\input{defn_hat_phi_sigma_gens_2.pstex_t}}
\end{array} \\ & \\ \hline \multicolumn{2}{|c|}{\hspace{1cm}}  \\
\multicolumn{2}{|c|}{
\begin{array}{c}
\scalebox{.4}{\input{gv_invariant_3.pstex_t}}
\end{array} \stackrel{\hat{\varphi}_{\Sigma}}{\longrightarrow} \begin{array}{c}
\scalebox{.4}{\input{defn_hat_phi_sigma_gens_3.pstex_t}}
\end{array}} \\ \multicolumn{2}{|c|}{\hspace{1cm}}  \\ \hline 
\end{array}
\]
\caption{Definition of $\hat{\varphi}_{\Sigma}$ on generations in $T(\Sigma)$.} \label{fig_defn_hat_varphi}
\end{figure}

\begin{lemma} $\hat{\varphi}_{\Sigma}$ descends to a map $\hat{\varphi}_{\Sigma}:\mathscr{A}(\Sigma) \to \mathscr{A}$.
\end{lemma} 
\begin{proof} Note that $\hat{\varphi}_{\Sigma}(T(\Sigma)) \subseteq T$. If one of $\alpha,\beta,\gamma$ is null homologous, then at least two of $\omega_{\alpha,\beta}$, $\omega_{\beta,\gamma}$ and $\omega_{\alpha,\gamma}$ are zero. Thus, $|\{\omega_{\alpha,\beta}, \omega_{\beta,\gamma},\omega_{\alpha,\gamma}\}|<3$. It follows that $\hat{\varphi}_{\Sigma}(R(\Sigma)) \subseteq R$. This proves the lemma.
\end{proof}

\begin{theorem} \label{grva_factors} The virtual knot invariant $\phi:\mathbb{Z}[\mathscr{K}] \to \mathscr{A}$ is the virtual knot analogue of the family $\{\Phi[\Sigma]:\mathbb{Z}[\mathscr{K}(\Sigma)] \to \mathscr{A}(\Sigma)| \Sigma \in \mathscr{S}\}$. In particular, we have that the following diagram commutes for all $\Sigma \in \mathscr{S}$.
\[
\xymatrix{
\mathbb{Z}[\mathscr{K}(\Sigma)] \ar[r]^{\Phi[\Sigma]} \ar[d]_{\varphi_{\Sigma}} & \mathscr{A}(\Sigma) \ar[d]^{\hat{\varphi}_{\Sigma}}\\
\mathbb{Z}[\mathscr{K}] \ar[r]_{\phi} & \mathscr{A} \\
}
\]    
\end{theorem}
\begin{proof} Let $(\Sigma,\tau)$ be a knot diagram on $\Sigma$ having Gauss diagram $D_{\tau}$. Let $D_{||}'$ denote a labeled subdiagram of $D_{\tau}$ having two non-intersecting arrows. Let $\delta$, $\epsilon$, $\zeta \in H_1(\Sigma;\mathbb{Z})$ be the homology classes corresponding to the labels. We must show that $\hat{\varphi}_{\Sigma}$ assigns the same relative weight to the region pairs of $D_{||}'$ as does $\phi$.
\newline
\newline
Suppose first that every arrow of $D_{\tau}$ has at most one endpoint in each of the regions of $D_{||}'$. Then $\delta$, $\epsilon$, $\zeta$ are represented by oriented one-dimensional submanifolds $M_{\delta}$, $M_{\epsilon}$, $M_{\zeta}$. In other words, $[M_{\delta}]=\delta$, $[M_{\epsilon}]=\epsilon$, and $[M_{\zeta}]=\zeta$. Since $\tau$ is a knot diagram on $\Sigma$, we must have that any two of $M_{\delta}$, $M_{\epsilon}$, $M_{\zeta}$ intersect transversally in a finite number of points.
\newline
\newline
By definition, we have that $\omega_{\delta,\epsilon}=|[\delta \times \epsilon,\mu']|$. Since, $M_{\delta}$ and $M_{\epsilon}$ are one-dimensional submanifolds of $\Sigma$, we have that:
\[
|M_{\delta} \cdot M_{\epsilon}|=|[\delta \times \epsilon,\mu']|,
\]
where $M_{\delta}\cdot M_{\epsilon}$ is the intersection number of $M_{\delta}$ and $M_{\epsilon}$ (see Proposition (31.7) \cite{gh}). 
\newline
\newline
Alternatively, the intersection number can be computed by adding up the local intersection numbers of the transversal intersections. Without loss of generality, we may assume that $\Sigma$ is oriented in such a way that the local intersection numbers are given as in Figure \ref{fig_loc_int_numbers}. The blue arc represents the arc of $M_{\delta}$ while the red arc represents the arc of $M_{\epsilon}$. Since $M_{\delta}$ and $M_{\epsilon}$ represent partial state curves of the knot diagram on $\Sigma$, each transversal intersection between $M_{\delta}$ and $M_{\epsilon}$ represents a classical crossing of $K$. Note that if the crossing is changed from $\oplus \to \ominus$ or $\ominus \to \oplus$, then the contribution of this intersection to the intersection number is the same.
\newline
\newline
Let $A_{\delta}$ be the region in the Gauss diagram labeled by $\delta$ and $A_{\epsilon}$ the region in the Gauss diagram labeled by $\epsilon$. Consider the computation of the weight $w_{\delta,\epsilon}$ of the pair of regions $A_{\delta}$ and $A_{\epsilon}$. Let $a_{\delta,\epsilon}$ denote an arrow used in defining the weight. Each arrow $c$ which intersects $a_{\delta,\epsilon}$ gives a contribution of $\pm 1$ to $w_{\delta,\epsilon}$. Moreover, the contribution of $c$ does not change if its direction and sign are both changed. Therefore, $\text{ind}_{a_{\delta,\epsilon}}(c)$ is $1$ or $-1$ according to whether $M_{\epsilon}$ crosses $M_{\delta}$ from left to right or right to left at the classical crossing corresponding to $c$. Thus, $w_{\delta,\epsilon}=|[\delta \times \epsilon,\mu']|=\omega_{\delta,\epsilon}$. Similarly, $w_{\epsilon,\zeta}=\omega_{\epsilon,\zeta}$ and $w_{\delta,\zeta}=\omega_{\delta,\zeta}$ when each arrow of $D_{\tau}$ has at most one endpoint in each of the regions of $\tau$.
\newline
\newline
Now suppose that $D_{\tau}$ has an arrow whose endpoints lie in one region of $D_{||}'$. Again, let $\delta,\epsilon,\zeta$ denote the homology classes in the labeling of $D_{||}'$. If $K$ is given the oriented smoothing at the crossings corresponding to the arrows of $D_{||}'$, then the result is a three component link diagram on $\Sigma$. Denote the thee components by $K_{\delta}$, $K_{\epsilon}$, $K_{\zeta}$ according to their homology class on $\Sigma$.
\newline
\newline
At least one of the diagrams $K_{\delta}$, $K_{\epsilon}, K_{\zeta}$, say $K_{\delta}$, must contain at least one of the (unsmoothed) crossings of $K$.  Apply the oriented smoothing at each of the crossing of $K_{\delta}$ to obtain $n_{\delta}$ pairwise disjoint closed curves $M_{\delta}^1,\ldots,M_{\delta}^{n_{\delta}}$. Similarly, we apply the oriented smoothing at each of the crossings of $K_{\epsilon}$ and $K_{\zeta}$ to obtain two families of pairwise disjoint closed curves $M_{\epsilon}^1,\ldots,M_{\epsilon}^{n_{\epsilon}}$ and $M_{\zeta}^1,\ldots,M_{\zeta}^{n_{\zeta}}$. Each of these oriented simple closed curves corresponds to a homology class $H_1(\Sigma;\mathbb{Z})$. Define $\delta_q=[M_{\delta}^q]$, $\epsilon_r=[M_{\epsilon}^r]$, and $\zeta_s=[M_{\zeta}^s]$ for all $q$, $r$, $s$. It is clear that $\delta=\sum \delta_q$, $\epsilon=\sum \epsilon_r$, and $\zeta=\sum \zeta_s$ in $H_1(\Sigma;\mathbb{Z})$. We now compute $[\delta \times \epsilon,\mu']$, $[\epsilon \times \zeta,\mu']$, and $[\delta \times \zeta,\mu']$:

\begin{eqnarray*}
[\delta \times \epsilon,\mu'] &=& \sum_{q,r} [\delta_q \times \epsilon_r,\mu'] \\
&=& \sum_{q,r} M_{\delta}^q \cdot M_{\epsilon}^r, \\
\end{eqnarray*}
\begin{eqnarray*}
[\epsilon \times \zeta,\mu'] &=& \sum_{r,s} M_{\epsilon}^q \cdot M_{\zeta}^s, \\
\end{eqnarray*}
\begin{eqnarray*}
[\delta\times\zeta,\mu'] &=& \sum_{q,s} M_{\delta}^q \cdot M_{\zeta}^s.
\end{eqnarray*}  
Consider the transversal intersections between $M_{\delta}^q$ and $M_{\epsilon}^r$. If there are no transversal intersections, then $M_{\delta}^q \cdot M_{\epsilon}^p=0$. The intersections of $M_{\delta}^q$ and $M_{\epsilon}^r$ correspond to arrows of the Gauss diagram where one endpoint is in $M_{\delta}^q$ and one endpoint is in $M_{\epsilon}^r$. Then one endpoint of the arrow must be in $M_{\delta}$ and one endpoint must be in $M_{\epsilon}$. Thus, we conclude that $|M_{\delta}^q\cdot M_{\epsilon}^r|$ can be computed by adding the local intersection numbers corresponding to the crossings as in Figure \ref{fig_loc_int_numbers}. It follows that $w_{\delta,\epsilon}=|[\delta \times \epsilon,\mu']|=\omega_{\delta,\epsilon}$, $w_{\epsilon,\zeta}=\omega_{\epsilon,\zeta}$ and $w_{\delta,\zeta}=\omega_{\delta,\zeta}$. This completes the proof.
\end{proof}

\begin{figure}
\[
\begin{array}{cc}
\begin{array}{c}
\scalebox{.35}{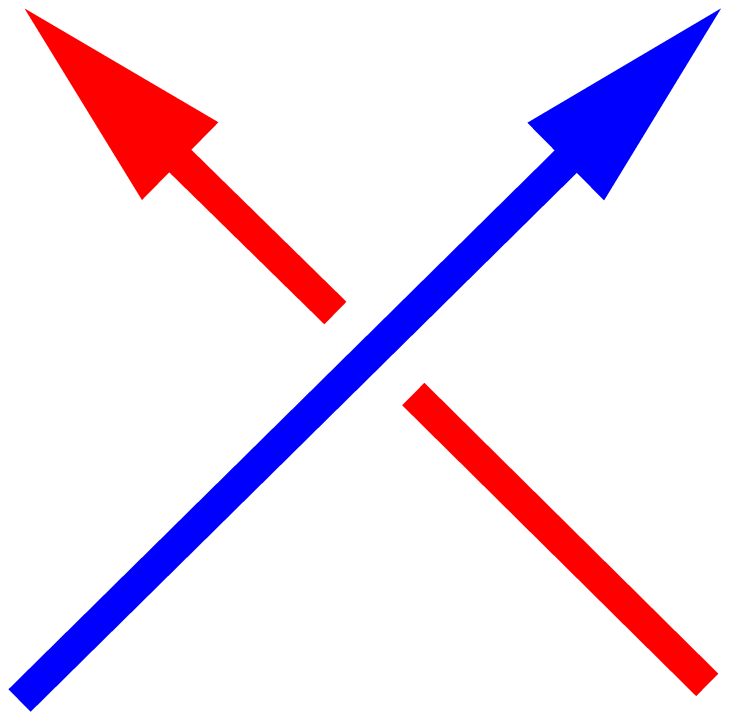} \\
\text{Contribution is } +1
\end{array}
& 
\begin{array}{c}
\scalebox{.35}{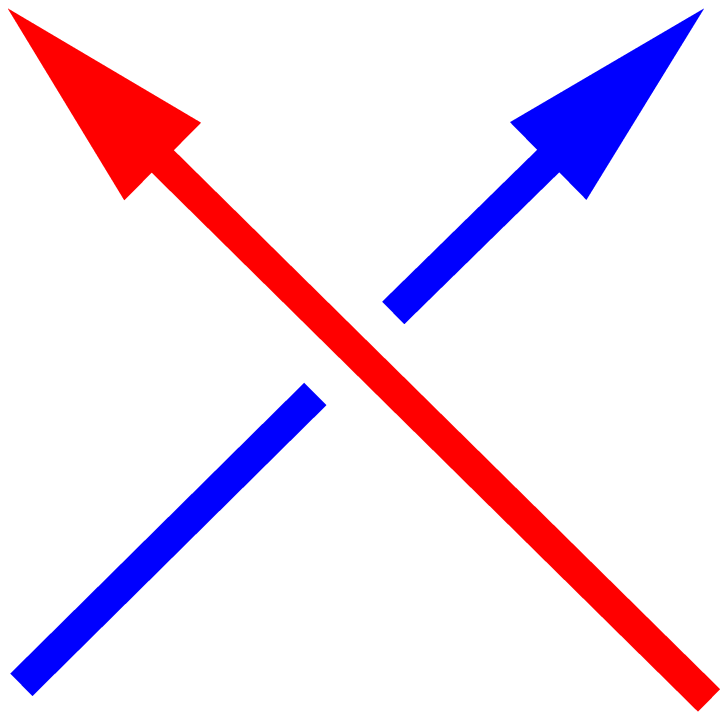} \\
\text{Contribution is } +1
\end{array} \\
\begin{array}{c}
\scalebox{.35}{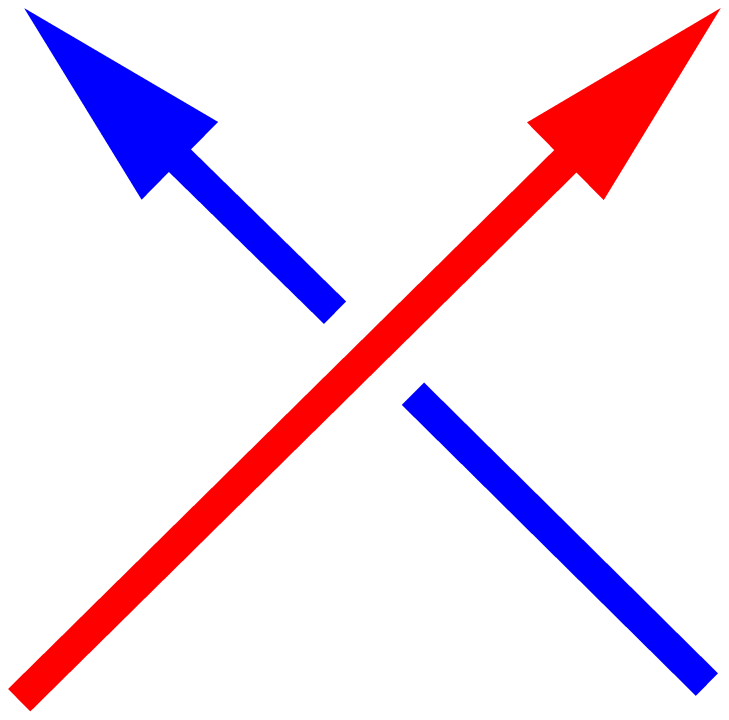} \\
\text{Contribution is } -1
\end{array}
& 
\begin{array}{c}
\scalebox{.35}{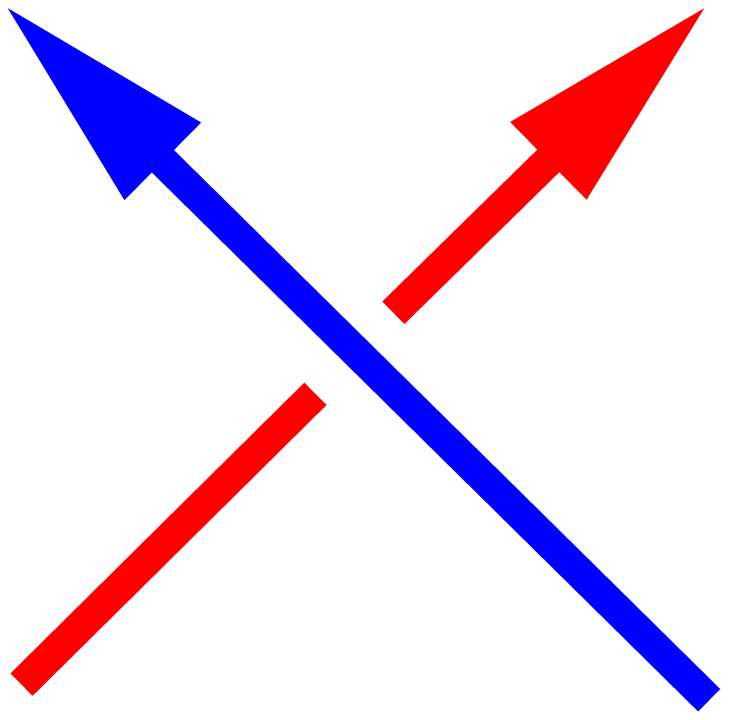} \\
\text{Contribution is } -1
\end{array} \\
\end{array}
\]
\caption{Computations of local topological intersection numbers for the four possibilities of classical crossings at a double point.} \label{fig_loc_int_numbers}
\end{figure}

\section{Concluding Remarks; A Problem} \label{sec_conc}

\noindent The existence of a universal finite-type invariant of classical knots has been long settled \cite{BN2}. Moreover, there is a universal finite-type invariant of knots in a given thickened surface \cite{AM,AMR}. Unfortunately, the existence of a universal finite-type invariant of virtual knots has been established only in the case of invariants of degree one \cite{allison}. 
\newline
\newline
The work of the present paper suggests the following approach to finding a universal finite-type invariant of virtual knots. Let $\Sigma$ be a closed oriented surface and let $\mu[\Sigma]:\mathbb{Z}[\mathscr{K}(\Sigma)] \to \mathscr{U}(\Sigma)$ denote a universal finite-type invariant of order $n$ for knots in $\Sigma \times I$. Here $\mathscr{U}(\Sigma)$ is some appropriate abelian group depending on $\Sigma$. For example, choose $\mu[\Sigma]$ according to the invariant given in \cite{AM} or choose $\mu[\Sigma]$ according to the invariant given in \cite{AMR}.
\newline
\newline
\centerline{
\fbox{\parbox{5.75in}{
\underline{\textbf{Problem}:} Find a virtual knot analogue of $\{\mu[\Sigma]:\mathbb{Z}[\mathscr{K}(\Sigma)] \to \mathscr{U}(\Sigma): \Sigma \in \mathscr{S}\}$. In particular, find a finite-type invariant of virtual knots $\mu:\mathbb{Z}[\mathscr{K}] \to \mathscr{U}$, where $\mathscr{U}$ is some abelian group, such that for all $\Sigma \in \mathscr{S}$ there is a map $\overline{\varphi}_{\Sigma}:\mathscr{U}(\Sigma) \to \mathscr{U}$ satisfying $\overline{\varphi}_{\Sigma} \circ \mu[\Sigma]=\mu \circ \varphi_{\Sigma}$. Prove that $\mu$ is a universal finite-type invariant of virtual knots.
}
}
}
\newline
\newline
\newline
The advantage to this approach is that it does not require one to ``invent'' a virtual knot invariant which might turn out to be universal.  Instead, we are borrowing universality from an invariant which is already known to be universal. Moreover, the structure of these invariants is fairly well understood.  The difficult part is to figure out how to capture the topological information in $\Sigma$ with a combinatorial structure for virtual knot diagrams. In this paper, intersection theory was used to make the connection to relative weights. Whether or not a similar approach will apply in full generality is unclear. 
\newline
\newline
A natural first place to start is to extend the arguments presented here to the Grishanov-Vassiliev finite-type invariants of order $n$. The authors strongly believe that such an extension holds but do not have a proof of it as of this writing.

\bibliographystyle{plain}
\bibliography{loop_bib}

\end{document}